\documentclass[a4paper]{article}
\usepackage{amsfonts}
\usepackage{mflogo,texnames}
\usepackage{amsmath,amssymb,amsthm}
\usepackage{color}
\usepackage[colorlinks=true]{hyperref}

\numberwithin{equation}{section}

\newtheorem{theorem}{Theorem}[section]
\newtheorem{lemma}{Lemma}[section]
\newtheorem{proposition}{Proposition}[section]

\newtheorem{corollary}{Corollary}[section]
\newtheorem{remark}{Remark}[section]

\hypersetup{linkcolor=blue,urlcolor=red,citecolor=red}
\allowdisplaybreaks[4] \numberwithin{equation}{section}

\begin{document}

\title{\textbf{Observable set, observability, interpolation inequality and spectral inequality for the heat equation in  $\mathbb{R}^n$ }}

\bigskip

\author{Gengsheng Wang\thanks{Center for Applied Mathematics, Tianjin University,
Tianjin, 300072, China. \emph{Email: wanggs62@yeah.net}}\quad
Ming Wang\thanks{School of Mathematics and Physics, China University of Geosciences, Wuhan, 430074,  China. \emph{Email: mwangcug@outlook.com}}\quad
Can Zhang\thanks{School of Mathematics and Statistics, Wuhan University, Wuhan, 430072, China,
and Department of Mathematics, University of the Basque Country (UPV/EHU),  Bilbao, 48080, Spain. \emph{Email: zhangcansx@163.com}}\quad
Yubiao Zhang\thanks{Center for Applied Mathematics, Tianjin University,
Tianjin, 300072, China. \emph{Email: yubiao\b{ }zhang@yeah.net}}
}

\maketitle

\begin{abstract}
This paper studies connections among observable sets, the observability inequality,
 the H\"{o}lder-type interpolation inequality and  the  spectral inequality
    for the heat equation in $\mathbb R^n$.  We
present  the characteristic  of observable sets for the  heat equation. In more detail, we show that    a measurable set in $\mathbb{R}^n$ satisfies the observability inequality  if and only if it
is $\gamma$-thick at scale $L$ for some $\gamma>0$ and $L>0$.
  We also build   up  the equivalence among the above-mentioned three inequalities.
   More precisely, we obtain that if a measurable set in $\mathbb{R}^n$ satisfies
   one of these  inequalities, then it satisfies others.  Finally, we get some weak  observability inequalities and                                                                                                                                 weak interpolation inequalities where  observations are made over a ball.
\end{abstract}

\medskip



\tableofcontents
\section{Introduction}
In this paper, we  consider  the heat equation in the whole physical space $\mathbb{R}^n$:
\begin{align}\label{heat}
\partial_tu -\triangle u=0\;\;\;\;\text{in}\;\;(0,\infty)\times\mathbb{R}^n, \quad u(0,\cdot)\in L^2(\mathbb{R}^n).
\end{align}
For this equation, we will characterize the observable sets  and build up
connections among several important inequalities which are introduced in the next  subsection.

\paragraph{Notation}  Write $C(\cdots)$ for a positive constant
 that depends on what are included in the brackets and may vary in different contexts.
 The  same  can be said about
  $C'(\cdots)$, $C_1(\cdots)$  and so on.  Use $V_n$ to denote  the volume of the unit ball in $\mathbb{R}^n$. Let $B_r(x)$, with  $x\in \mathbb{R}^n$ and $r>0$,
   be the open ball in $\mathbb{R}^n$, centered at $x$   and of radius $r$. (Simply write $B_r=B_r(0)$.) Let $\mathbb{S}^{n-1}$ be the unit spherical surface in $\mathbb{R}^n$.
   Let $\mathbb{N}:=\{0,1,2,\dots\}$.
   Denote by $Q$ the open unit  cube in $\mathbb{R}^n$, centered at the origin. Let $x+LQ$, with $x\in \mathbb{R}^n$ and $L>0$, be the set $\{x+Ly: y\in Q\}$.
   For each measurable set $D\subset\mathbb R^n$, denote by $|D|$ 
   and  $D^c$    
   its Lebesgue measure and  complement set respectively.
   For any set $G$, we write $\chi_G$ for the characteristic function of $G$.
   Given $f\in L^2(\mathbb{R}^n)$, write $\widehat{f}$ for its Fourier transform\footnote{
   Given $f$ in the Schwartz class $\mathcal{S}(\mathbb R^n)$, its  Fourier transform is as:
$\widehat{f}(\xi) = (2\pi)^{-n/2}\int_{\mathbb{R}^n}e^{-ix\xi}f(x)\,\mathrm dx$, $\xi\in\mathbb R^n$.  Since $\mathcal{S}(\mathbb R^n)$ is dense in $L^2(\mathbb{R}^n)$,
by a standard way, we can define $\widehat{f}$ for each $f\in L^2(\mathbb{R}^n)$.}.
Given a measurable function $f$ over $\mathbb{R}^n$, we denote by $supp\; f$ 
the support of $f$, which is the set of all points (in $\mathbb{R}^n$) where $f$ does not vanish.
Let $\{e^{t\triangle} : t\geq 0\}$ be the semigroup
generated by the Laplacian operator in $\mathbb R^n$.
Given $x=(x_1,\dots,x_n)\in \mathbb{R}^n$, let $|x|:=\left(\sum_{i=1}^nx_i^2\right)^{1/2}$ and $\langle x\rangle:= \sqrt{1+|x|^2}$.

\subsection{Thick sets and several  inequalities}

We start with introducing  sets of $\gamma$-thickness at scale $L$.

\paragraph{Sets of $\gamma$-thickness at scale $L$}

 {\it A measurable set $E\subset\mathbb{R}^n$ is
said to be $\gamma$-thick at scale $L$ for some $\gamma>0$ and $L>0$, if
\begin{align}\label{equ-set}
\left|E \bigcap (x+LQ)\right|\geq \gamma L^n\;\;\mbox{for each}\;\;x\in \mathbb{R}^n.
\end{align}}
 About  sets of $\gamma$-thickness at scale $L$, several remarks are given in order.
\begin{description}
\item[($a_1$)] To our best knowledge, this definition arose from studies of the uncertainty principle.   We quote it from \cite{BD} (see Page 5 in \cite{BD}). Before \cite{BD},
    some very similar concepts were proposed. For instance, the definition of {\it relative dense sets} was given in \cite{Ka}  (see also Page 113 in \cite{HJ}); the definition of {\it thick sets} was introduced in
    \cite{K}.

    \item[($a_2$)] Each set $E$ of $\gamma$-thickness at scale $L$ has the following properties: First,  in each cube with the length $L$, $|E|$   is bigger than or equals to  $\gamma L^n$.
  Second, $E$
 is  also a set of $\gamma$-thickness at scale $2L$, but the reverse is not true. Third,
 we  necessarily have that $\gamma\leq 1$.
 \end{description}

\vskip 10pt

Next, we introduce an observability inequality for the equation (\ref{heat}).
\paragraph{The observability inequality}
{\it A measurable set $E\subset\mathbb{R}^n$ is said to  satisfy the observability
inequality
 for the equation (\ref{heat}), if for any $T>0$ there exists a positive constant $C_{obs}=C_{obs}(n,T,E)$ so that when $u$ solves (\ref{heat}),
\begin{align}\label{ob-q}
\int_{\mathbb{R}^n}|u(T,x)|^2\,\mathrm dx\leq C_{obs}\int_0^T\int_E|u(t,x)|^2\,\mathrm dx\,\mathrm dt.
\end{align}
When a measurable $E\subset\mathbb{R}^n$ satisfies (\ref{ob-q}), it is called an observable set for (\ref{heat}). }

\vskip 5pt

 Several notes on the observability
inequality \eqref{ob-q} are given in order.

\begin{description}
\item[($b_1$)] By  treating the integral on  the left hand side as a recovering term, and the integral on the right hand side as an observation term, we can understand the inequality \eqref{ob-q} as follows:  one can recover a solution of \eqref{heat}
at time $T$, through observing it on the set $E$ and in the time interval $(0,T)$.
From perspective of  control theory, the inequality \eqref{ob-q} is equivalent to
 the following null controllability: {\it
 For any $u_0 \in L^2(\mathbb{R}^n)$ and $T>0$, there exists a control $f\in L^2((0,T)\times \mathbb R^n)$ driving  the solution $u$ to the controlled equation:
 $\partial_tu -\triangle u=\chi_{E}f$ in $(0,T)\times\mathbb R^n$, from the initial state $u_0$ to the state $0$ at time $T$. }

\item[($b_2$)] We can compare \eqref{ob-q} with the observability inequality for the heat equation on a bounded physical domain.
Let $\Omega$
  be a bounded $C^2$ (or Lipschitz and locally star-shaped, see \cite{AEWZ}) domain in $\mathbb{R}^n$. Consider the equation:
\begin{align}\label{1.3WSGENG}
\begin{cases}
\partial_tu -\triangle u=0\;\;\;\;\;\;\;\;\quad \text{in}\;\;(0,\infty)\times\Omega,\\
 u=0 \qquad \qquad\quad\;\;\;\;\;\; \text{on}\;\;(0,\infty)\times\partial\Omega,\\
 u(0,\cdot)\in L^2(\Omega).
\end{cases}
\end{align}
{\it We say that a measurable set $\omega\subset\Omega$
satisfies the observability inequality for (\ref{1.3WSGENG}), if given
$T>0$, there is a constant $C(n, T, \omega,\Omega)$ so that when $u$ solves (\ref{1.3WSGENG}),
\begin{align}\label{ob-heat-open}
\int_\Omega |u(T,x)|^2 \,\mathrm dx\leq C(n, T, \omega,\Omega)\int_0^T\int_\omega |u(t,x)|^2 \,\mathrm dx \,\mathrm dt.
\end{align}
When a measurable set $\omega\subset\Omega$ satisfies (\ref{ob-heat-open}), it is called
an observable set for (\ref{1.3WSGENG}).
 }

The inequality (\ref{ob-heat-open}) has been widely studied.  See \cite{FI, LR, LZZ}
for the case where $\omega$ is open; \cite{AE, AEWZ, EMZ} for the case when $\omega$ is measurable.

\item[($b_3$)] When $\Omega$ is an unbounded domain and $\omega$ is a bounded and open subset of $\Omega$, the inequality  (\ref{ob-heat-open}) may not be true.
    This was showed in \cite{MZ} for the heat equation in the physical domain $\mathbb{R}^+$.
      Similar results have been obtained for higher dimension cases in \cite{MZb}.
 For the heat equation  in an unbounded domain,
  \cite{M05a} imposed
  a condition, in terms of the Gaussian kernel, on  the  set $\omega$  so that the observability inequality does not hold. In particular,  \cite{M05a} showed that the observability inequality  fails when $\Omega$ is unbounded and $|\omega|<\infty$. Notice
  that   any set $E\subset \mathbb{R}^n$ of finite measure does not have the characteristic on observable
  sets of (\ref{heat}). {\it This characteristic is indeed  the $\gamma$-thickness at scale $L$ for some $\gamma>0$ and $L>0$.} (See Theorem~\ref{equi-thm} of this paper.)

  About works on sufficient conditions of observable sets  for heat equations in  unbounded domains, we would like to mention  the work \cite{CMZ}. It showed that, for some
 parabolic equations in an  unbounded domain $\Omega \subset \mathbb{R}^n$,
  the observability inequality holds when  observations are made over a subset $E\subset\Omega$, with $\Omega \backslash E$ bounded. For other similar results, we refer the reader to \cite{M05b}. When $\Omega=\mathbb{R}^n$, such a set $E$ has the characteristic on observable sets of (\ref{heat}) mentioned before.

  \item[($b_4$)] An interesting phenomenon is that some potentials (growing at infinity) in heat equations may change the above-mentioned characteristic
  on observable sets for the heat equations with potentials.
   In \cite{M09, DM}, the authors realized the following fact: Let $A=\triangle+V$, where $V(x):=-|x|^{2k}$, $x\in \mathbb{R}^{n}$, with $2\leq k\in \mathbb{N}$.  Write $\left\{e^{tA}\right\}_{t\geq0}$ for the semigroup generated by the operator $A$.
Let $r_0\geq 0$ and let $\Theta_0$ be an open subset of $\mathbb{S}^{n-1}$. Let $\Gamma=\{x\in \mathbb{R}^{n}: |x|\geq r_0, x/|x|\in \Theta_0\}$. Then there is
 $C(n, T, \Theta_0,r_0, k)$ so that
      \begin{align}\label{ob-heat-potential}
      \int_{\mathbb{R}^{n}}\left|e^{TA}u_0\right|^2\,\mathrm dx\leq C(n, T, \Theta_0,r_0, k)\int_0^T\int_{\Gamma}\left|e^{tA}u_0\right|^2\,\mathrm dx\,\mathrm dt\;\; \mbox{for all}\;\; u_0\in L^2(\mathbb{R}^{n}).
      \end{align}
          The cone $\Gamma$ does not have the characteristic on  observable sets mentioned before,  but  still holds the observability inequality (\ref{ob-heat-potential}).    The main reason is as follows: The unbounded potential
 $V$ changes the behaviour of the solution of the pure heat equation (\ref{heat}).
 This
 plays an important role in the proof of \eqref{ob-heat-potential} (see \cite{M09, DM}). It should be pointed out that when $V(x)=-|x|^2$, $x\in \mathbb{R}^{n}$ (which means that the potential grows more slowly  at infinity),  \eqref{ob-heat-potential} does not hold for the above cone. We refer the readers to \cite{M09, DM} for more details on this issue.
 Besides, we also would like to mention \cite{B} for this subject.

   {\it An interesting question now arises: How do potentials influence characteristics of observable sets?} We wish to answer this question in our future studies.

  \end{description}
\vskip 10pt

We then introduce an interpolation  inequality for the equation (\ref{heat}).

\paragraph{The H\"older-type interpolation inequality}
{\it A measurable set  $E\subset\mathbb R^n$ is said to  satisfy the H\"older-type interpolation inequality for the heat equation  (\ref{heat}),  if for any  $\theta\in (0,1)$, there is $C_{Hold}=C_{Hold}(n,E,\theta)$ so that for each $T>0$ and each solution $u$ to the equation \eqref{heat},
\begin{align}\label{interpolation}
\int_{\mathbb{R}^n}|u(T,x)|^2\,\mathrm dx\leq e^{C_{Hold}(1+\frac{1}{T})}\Big( \int_{E}|u(T,x)|^2\,\mathrm dx\Big)^\theta\Big( \int_{\mathbb{R}^n}|u(0,x)|^2\,\mathrm dx\Big)^{1-\theta}.
\end{align}}
Several remarks on the H\"older-type interpolation inequality (\ref{interpolation}) are given in order.

\begin{description}
\item[$(c_1)$] The above H\"older-type interpolation inequality
is equivalent to what follows: There is $\theta=\theta(n,E)\in (0,1)$ and $C_{Hold}=C_{Hold}(n,E)$ so that \eqref{interpolation}
holds for all $T>0$ and solutions $u$ to \eqref{heat}. This can be verified by the
 similar way used in the proof of   \cite[Theorem 2.1]{PWX}.

\item[$(c_2)$] The inequality \eqref{interpolation} is a kind of quantitative unique continuation for the heat equation (\ref{heat}).
    It provides a H\"{o}lder-type propagation of smallness for solutions of the heat equation (\ref{heat}). In fact, if $ \displaystyle\int_{E}|u(T,x)|^2\,\mathrm dx=\delta$, then we derive from \eqref{interpolation} that $\displaystyle\int_{\mathbb{R}^n}|u(T,x)|^2\,\mathrm dx$ is bounded by $C\delta^\theta$ for some constant $C>0$. Consequently, $u(T,\cdot)=0$
    over $\mathbb{R}^n$ provided that it is zero over $E$.

    \item[$(c_3)$] From perspective of control theory, the inequality \eqref{interpolation} implies the approximate null controllability
        with cost for  impulse
        controlled heat equations, i.e.,  {\it given $T>\tau>0$, $\varepsilon>0$,
        there is $C=C(n,E,  T,\tau,\varepsilon)$ such that
        for each
        $u_0\in L^2(\mathbb{R}^n)$, there is $f\in  L^2(\mathbb{R}^n)$  so that
        $$
        \|f\|_{L^2(\mathbb{R}^n)}\leq C\|u_0\|_{L^2(\mathbb{R}^n)}\;\;\mbox{and}\;\;
        \|u(T,\cdot)\|_{L^2(\mathbb{R}^n)}\leq \varepsilon\|u_0\|_{L^2(\mathbb{R}^n)},
        $$
        where $u$ is the solution to
        the impulse controlled equation: $\partial_t u-\Delta u=\delta_{\{t=\tau\}}\chi_E f$ in $(0,T)\times \mathbb{R}^n$, with the initial condition $u(0,x)=u_0(x)$, $x\in  \mathbb{R}^n$.}        (See \cite[Theorem 3.1]{PWX}.)
        \item[$(c_4)$] The H\"older-type interpolation inequality (\ref{interpolation}) can imply the observability inequality \eqref{ob-q}.
            Moreover, it leads to the following stronger version of  \eqref{ob-q}:
\begin{align}\label{wangobs1.8}
\int_{\mathbb{R}^n}|u(T,x)|^2\,\mathrm dx\leq C_{obs}\int_F\int_E|u(t,x)|^2\,\mathrm dx\,\mathrm dt, \;\;\mbox{with}\;\;C_{obs}=C_{obs}(n,T,E,F),
\end{align}
where $F\subset (0,T)$ is a subset of positive measure. This will be presented in Lemma
\ref{lem-obs}.
We derive (\ref{wangobs1.8})
from (\ref{interpolation}), through using the telescoping series method developed in
\cite{PW} (see also \cite{PWZ, AEWZ})
for heat equations in bounded domains.

\item[$(c_5$)] We can compare  (\ref{interpolation}) with an interpolation inequality
for the heat equation \eqref{1.3WSGENG}.
 {\it A measurable set $\omega\subset\Omega$  is said to  satisfy the H\"older-type interpolation inequality
 for the equation (\ref{1.3WSGENG}),
 if for any $\theta\in(0,1)$, there is $C=C(n,  \Omega, \omega, \theta)$ so that for any $T>0$ and any solution $u$ to  \eqref{1.3WSGENG},
    \begin{align}\label{1.7GGWWSSG}
\int_{\Omega}|u(T,x)|^2\,\mathrm dx\leq e^{C(1+\frac{1}{T})}\Big( \int_{\omega}|u(T,x)|^2\,\mathrm dx\Big)^\theta\Big( \int_{\Omega}|u(0,x)|^2\,\mathrm dx\Big)^{1-\theta}.
\end{align}}
In \cite{PW10}, the authors proved that any open and nonempty subset $\omega\subset\Omega$ satisfies  the H\"older-type interpolation inequality (\ref{1.7GGWWSSG})
for heat equations with potentials in  bounded and convex domains. The frequency function method used in   \cite{PW10} was partially borrowed from \cite{EFV}.
In \cite{AEWZ}, the authors proved that any subset $\omega$ of positive measure satisfies the H\"older-type interpolation inequality (\ref{1.7GGWWSSG})
for the heat equation (\ref{1.3WSGENG}) where $\Omega$ is a bounded
Lipschitz and locally star-shaped domain in $\mathbb{R}^n$.
More about this inequality for heat equations in  bounded domains, we referee the readers to \cite{PW, PWX, PWZ}.

\end{description}
\vskip 10pt

Finally, we will introduce a spectral  inequality for some functions in $L^2(\mathbb{R}^n)$.

\paragraph{The spectral inequality}
{\it A measurable set  $E\subset\mathbb{R}^n$ is said to  satisfy the spectral inequality, if
 there is a positive constant $C_{spec}=C_{spec}(n,E)$ so that for each $N>0$,
\begin{align}\label{intro-spec}
\int_{\mathbb{R}^n}|f(x)|^2 \,\mathrm dx \leq e^{C_{spec}(1+N)}\int_E |f(x)|^2 \,\mathrm dx
\;\;\mbox{for all}\;\;f\in L^2(\mathbb{R}^n)\;\;\mbox{with}\;\;supp \;\widehat{f}\subset B_N.
\end{align}}

\vskip 5pt

Several notes on the spectral inequality (\ref{intro-spec}) are given in order.
\begin{description}
\item[$(d_1)$]  Recall the Lebeau-Robbiano spectral inequality (see \cite{LR, LZUA}):
    {\it Let $\Omega$ be a bounded smooth  domain in $\mathbb{R}^n$ and let $\omega$ be a nonempty open subset of $\Omega$.
      Write $\triangle_\Omega$ for the Laplacian operator on $L^2(\Omega)$ with $Domain(\triangle_\Omega)=H^1_0(\Omega)\bigcap H^2(\Omega)$.
      Let $\{\lambda_j\}_{j\geq1}$ (with $0<\lambda_1<\lambda_2\leq\cdots$) be the eigenvalues
of $-\triangle_\Omega$ and let $\{\phi_j\}_{j\geq1}$ be the corresponding eigenfunctions.
      Then there is a positive constant $C(n,\Omega,\omega)$ so that
      for each $\lambda>0$,
\begin{align}\label{spec-bound}
\int_{\Omega}|f(x)|^2 \,\mathrm dx \leq e^{C(n,\Omega,\omega)(1+\sqrt{\lambda})}\int_{\omega}|f(x)|^2 \,\mathrm dx \;\;\mbox{for all}\;\; f\in span\{\phi_j\;:\; \lambda_j<\lambda\}.
\end{align}}
This inequality was extended to the case where $\Omega$ is a bounded $C^2$ domain via
a simpler way
in \cite{LVQ1}.
Then it  was extended to the case that $\Omega$ is  a bounded Lipschitz and locally star-shaped domain; $\omega$ is a subset of positive measure so that $\omega\subset B_R(x_0)\subset
B_{4R}(x_0)\subset\Omega$ for some $R>0$ and $x_0\in \Omega$;
  and $C(n,\Omega,\omega)=C(n,\Omega, |\omega|/|B_R|)$  (see \cite[Theorem 5 and Theorem 3]
   {AEWZ}).

By our understanding, the inequality (\ref{intro-spec}) is comparable to (\ref{spec-bound})
 from two perspectives as follows:
 First,  the inequality (\ref{intro-spec})
 is satisfied by functions in the subspace:
 $$
 E_N\triangleq \Big\{f\in L^2(\mathbb{R}^n)\; :\; supp \; \widehat{f}\subset B_N\Big\}
 \;\;\mbox{ with}\;\;N>0,
  $$
  while the inequality (\ref{spec-bound}) is satisfied by functions in the subspace:
 $$
 F_\lambda\triangleq\Big\{\sum_{\lambda_j<\lambda}f=a_j\phi_j\in L^2(\Omega)\;:\; \{a_j\}_{j\geq 1}\subset   \mathbb{R}\Big\}\;\;\mbox{ with}\;\;\lambda>0.
  $$
  From the definition of the spectral projection in the abstract setting given in \cite{SIMON1}
  (see Pages 262-263 in \cite{SIMON1}), we can define  two spectral projections: $\chi_{[0,N^2)}(-\Delta)$
  and $\chi_{[0,\lambda)}(-\Delta_\Omega)$  on $L^2(\mathbb{R}^n)$ and $L^2(\Omega)$,
  respectively.
  Then after some computations, we find that $E_N$ and $F_{\lambda}$ are  the ranges of  $\chi_{[0,N^2)}(-\Delta)$ and  $\chi_{[0,\lambda)}(-\Delta_\Omega)$, respectively.
 Second, the square root of the integral of $\chi_{[0,N^2)}$ over $\mathbb{R}$ is $N$
 which corresponds to the $N$ in (\ref{intro-spec}), while the square root of the integral of $\chi_{[0,\lambda)}$ over $\mathbb{R}$ is $\sqrt{\lambda}$ which corresponds to
 the $\sqrt{\lambda}$ in (\ref{spec-bound}).

\item[$(d_2)$]
 Though the inequality (\ref{intro-spec}) was first named as the spectral inequality in \cite{RM} (to our best knowledge), it has been extensively studied for long time. (See, for instance, \cite{BD, HJ, Ka, K, N, LS,  BPP1, BPP2,   Re}.) In \cite{K}, the author announced that if $E$ is $\gamma$-thick at scale $L$ for some $\gamma>0$ and $L>0$, then $E$ satisfies the spectral inequality (\ref{intro-spec}),
  and further proved this announcement for the case when $n=1$. Earlier,
  the authors of  \cite{LS} (see also \cite{HJ}) proved that $E$ is $\gamma$-thick at scale $L$
 for some $\gamma>0$ and $L>0$ if and only if  $E$ satisfies the following inequality: For each $N>0$, there is a positive constant $C(n,E,N)$ so that
\begin{align}\label{spec-LS}
\int_{\mathbb{R}^n}|f(x)|^2 \,\mathrm dx \leq C(n,E,N)\int_E |f(x)|^2 \,\mathrm dx
\;\;\mbox{for each}\;\;f\in L^2(\mathbb{R}^n)\;\;\mbox{with}\;\;supp \; \widehat{f}\subset B_N.
\end{align}
This result  is often referred as the Logvinenko-Sereda theorem.
Before \cite{LS}, the above equivalence was proved  by B. P. Paneyah for the case
that $n=1$ (see \cite{BPP2, BPP1, HJ}).
In \cite{Ka}, the author claimed (\ref{spec-LS}), with $C(n,E,N)=e^{C_{spec}(1+N)}$,
and proved this claim for the case when $n=1$.
  In the proof of our main theorem of this paper, the expression $C(n,E,N)=e^{C_{spec}(1+N)}$ will play an important role.
From this point of view,  (\ref{spec-LS}) is weaker than the spectral inequality  (\ref{intro-spec}).

\item[$(d_3)$] The inequality (\ref{spec-LS}) is also important. It is closely related  to the uncertainty principle (which is an extensive research topic in the theory of harmonic analysis and says roughly that a nonzero function and its Fourier transform cannot be both sharply localized, see \cite{FS}). In fact,
    a measurable set $E$  satisfies
    the inequality \eqref{spec-LS} if and only if it satisfies the following uncertainty principle:
\begin{align*}
\int_{\mathbb{R}^n}|f(x)|^2 \,\mathrm dx \leq C'(n,E,N)\left(\int_E |f(x)|^2 \,\mathrm dx + \int_{B_N^c}|\widehat{f}(\xi)|^2 \,\mathrm d\xi \right)\;\;\mbox{for all}\;\; f \in L^2(\mathbb{R}^n).
\end{align*}
We refer the interested readers to \cite{BD, HJ, P.Jaming, N} for the proof of the above result,
as well as  more general uncertainty principle, where  $E$ and $B_N^c$ are replaced by more general sets.

It deserves mentioning what follows: The uncertainty principle
 can help us to get the exact controllability for the  Schr\"{o}dinger equation with controls located outside of two balls and at two time points. This was realized in  \cite{WWZ}. (See \cite{HSF} for more general cases.)

\item[$(d_4)$] By using a global Carleman estimate, the authors in  \cite{RM} proved
the  spectral inequality  \eqref{intro-spec} for such an open subset $E$ that satisfies the property: there exists $\delta>0$ and $r>0$ so that
\begin{align}\label{thick-open}
\forall y \in \mathbb{R}^n, \exists \,y'\in E\textmd{ such that } B_r(y') \subset E\textmd{ and }|y-y'|\leq \delta.
\end{align}
It is clear that a set with the above property (\ref{thick-open}) is   a set of $\gamma$-thick at scale $L$ for some $\gamma>0$ and $L>0$\footnote{In fact, one can choose $L=2(\delta+r)$, $\gamma = r^n\left(2(\delta+r)\right)^{-n}V_n$.}.

\item[$(d_5)$] With the aid of the spectral inequality (\ref{intro-spec}), one can use 
the same strategy given in \cite{LR} to derive the null controllability described in the note $(b_1)$. 

\end{description}

\subsection{Aim, motivation and main result}
{\bf Aim}  According to the note $(d_2)$  in the previous subsection, the characteristic of a measurable set holding
the spectral inequality \eqref{intro-spec} is the $\gamma$-thickness at scale $L$
 for some $\gamma>0$ and $L>0$. Natural and interesting questions are as follows:
 What is the characteristic of  observable sets for (\ref{heat})?  How to characterize  a measurable set $E$ satisfying the H\"older-type interpolation inequality \eqref{interpolation}? What are the connections among inequalities (\ref{ob-q}), \eqref{interpolation} and \eqref{intro-spec}?
The  aim of this paper is to answer the above questions.
\vskip 5pt

\noindent{\bf Motivation} The motivations of our studies are given in order.
\begin{description}
\item[$(i)$]
 The first  motivation  arises from two papers \cite{BLR} and \cite{AE}. In \cite{BLR},
 the authors gave, for the wave equation in a bounded physical domain $\Omega\subset\mathbb R^n$,
 a sufficient and almost necessary condition
to ensure an open subset $\Gamma\subset\partial\Omega$ to be observable, (i.e., $\Gamma$ satisfies the observability inequality for the wave equation with  observations on $\Gamma$).
This condition is exactly the well known Geometric Control Condition (GCC for short)\footnote{An open subset $\omega \subset \Omega$ is said to satisfy the GCC if there exists $T_0>0$ such that any geodesic with velocity one  meets $\omega$ within time $T_0$ (see e.g. \cite{Lau}).}. Thus, we can say that the GCC condition is a characteristic of observable open sets on $\partial\Omega$, though  this condition is not strictly necessary (see \cite{RLT}).
   The authors in \cite{AE} presented a sufficient and necessary condition to ensure a measurable subset $\omega\subset\Omega$ satisfying (\ref{ob-heat-open}).
   This condition is as: $|\omega|>0$.
 Hence,   the characteristic of  observable  sets for the equation (\ref{1.3WSGENG})
 is as:
$|\omega|>0$.

 Analogically,  it should be very important to characterize observable sets for
 the heat equation (\ref{heat}). However,  it seems for us  that there is no any
 such result in the past publications. These motivate us to find the characteristic
 of observable sets for the equation (\ref{heat}).

\item[$(ii)$]  For the heat equation (\ref{1.3WSGENG}), the observability inequality (\ref{ob-heat-open}), the H\"older-type interpolation inequality (\ref{1.7GGWWSSG}) and the spectral inequality (\ref{spec-bound})
    are equivalent. More precisely, we have that if $\omega\subset\Omega$ is a measurable set, then
    \begin{eqnarray}\label{1.13GGWWSS}
    |\omega|>0\Longleftrightarrow \omega\;\mbox{satisfies}\;(\ref{spec-bound})\Longleftrightarrow
    \omega\;\mbox{satisfies}\;(\ref{1.7GGWWSSG})\Longleftrightarrow
    \omega\;\mbox{satisfies}\;(\ref{ob-heat-open}).
    \end{eqnarray}
The proof of (\ref{1.13GGWWSS}) was hidden in the paper \cite{AEWZ}. (See Theorem 5, Theorem 6, as well as its proof, Theorem 1, as well as its proof,  in \cite{AEWZ}.)
However,  for the heat equation (\ref{heat}), the equivalence among these three inequalities
has not been touched upon.
These  motivate
 us to build up the equivalence among inequalities (\ref{ob-q}), \eqref{interpolation} and \eqref{intro-spec}.

It deserves mentioning that for heat equations with lower terms in bounded physical domains,  we do not know if  (\ref{1.13GGWWSS}) is  true.
\end{description}

\noindent {\bf Main Result} The main result of the paper is the next Theorem~\ref{equi-thm}.
\begin{theorem}\label{equi-thm}
Let $E\subset \mathbb{R}^n$ be a measurable subset. Then the following statements are equivalent:
\begin{description}
  \item[(i)] The set $E$ is $\gamma$-thick at scale $L$ for some $\gamma>0$ and $L>0$.
  \item[(ii)] The set $E$ satisfies the spectral inequality (\ref{intro-spec}).
    \item[(iii)] The set $E$ satisfies the H\"older-type interpolation inequality \eqref{interpolation}.
  \item[(iv)] The set $E$ satisfies  the observability inequality (\ref{ob-q}).
\end{description}
\end{theorem}
Several remarks about Theorem~\ref{equi-thm} are given in order.
\begin{description}
 \item[$(e_1)$]
The equivalence of statements (i) and (iv) in Theorem~\ref{equi-thm} tells us:
 the characteristic
of observable sets for the heat equation (\ref{heat}) is
 the $\gamma$-thickness at scale $L$ for some $\gamma>0$ and $L>0$.
This seems to be new for us.
\item[$(e_2)$] The equivalence among statements (ii), (iii) and (iv) in Theorem~\ref{equi-thm}
presents  closed connections of the three inequalities. This seems also to be new for us.
\item[$(e_3)$] We find the following way to prove Theorem~\ref{equi-thm}:
${\bf (i) \Rightarrow (ii) \Rightarrow (iii) \Rightarrow (iv) \Rightarrow (i).} $
We prove ${\bf (i) \Rightarrow (ii)}$ by some ideas from \cite{K}. Indeed, this result
was announced in \cite{K} and then proved for the case that $n=1$ in the same reference. We prove ${\bf  (ii) \Rightarrow (iii) \Rightarrow (iv)}$, though using some ideas and techniques from  \cite{AEWZ, PW}. Finally, we show ${\bf (iv) \Rightarrow (i)} $ via the structure of a special solution to the equation (\ref{heat}).

\item[$(e_4)$] We noticed that four days after
we put our current work in arXiv, the paper \cite{EV17} appeared in arXiv. In \cite{EV17}, the authors independently got the equivalence (i) and (iv) in Theorem~\ref{equi-thm}.
 \end{description}

\subsection{Extensions to bounded observable sets}

From Theorem~\ref{equi-thm}, we see that in order to have
(\ref{ob-q}) or \eqref{interpolation}, the set $E$  has to be $\gamma$-thick at scale $L$ for some $\gamma>0$ and $L>0$. Then a natural and interesting question arises: What
are possible substitutions of (\ref{ob-q}) or \eqref{interpolation}, when $E$ is replaced by a  ball in $\mathbb{R}^n$? (It deserves to mention that any ball in $\mathbb{R}^n$ does not satisfy the thick condition \eqref{equ-set}.)
We try to find
the substitutes from two perspectives as follows:
 \begin{description}
 \item[(i)] We try to add weights on the left hand side
and
ask ourself if the following inequalities hold for all solutions of \eqref{heat}:
\begin{align}\label{ob-ball-1}
\int_{\mathbb{R}^n}\chi_{B_{r'}}(x)|u(T,x)|^2\,\mathrm dx\leq C(T,r',r,n)\int_0^T\int_{B_r}|u(t,x)|^2\,\mathrm dx\,\mathrm dt
\end{align}
and
\begin{align}\label{ob-ball-2}
\int_{\mathbb{R}^n}\rho(x)|u(T,x)|^2\,\mathrm dx\leq C(T,\rho,r,n)\int_0^T\int_{B_r}|u
(t,x)|^2\,\mathrm dx\,\mathrm dt,
\end{align}
where  $\rho(x)=\langle x\rangle^{-\nu}$ or $e^{-|x|}$.
On one hand,
 we proved that (\ref{ob-ball-1}) is true when  $r'<r$, while (\ref{ob-ball-1}) is not true when $r'>r$ (see Theorem \ref{prop-1} in Subsection \ref{07bao1}). Unfortunately, we do not know if (\ref{ob-ball-1}) holds when $r'=r$. On the other hand,
 we showed that (\ref{ob-ball-2}) fails for all  $r>0$ (see Corollary~\ref{corollary3.2} in Subsection \ref{07bao1}).

\item[(ii)] We try to find a class of initial data so that
(\ref{ob-q}) (where $E$ is replaced by $B_r$) holds for all solutions of (\ref{heat}) with initial data in this class. We have obtained some results on this issue (see Theorem \ref{obs-special-data} in Subsection \ref{07bao1}). {\it More interesting question is as: what is the biggest class of initial data so that
(\ref{ob-q}) (where $E$ is replaced by $B_r$) holds for all solutions of the heat equation (\ref{1.3WSGENG}) with initial data in this class? }
Unfortunately, we are not able to answer it.

We now turn to  possible substitutions of \eqref{interpolation} where $E$ is replaced by $B_1$. We expect to find $b(\varepsilon)>0$ for each $\varepsilon\in(0,1)$ so that for any $T>0$,
there is
 $C(n,T)>0$ such that when $u$ solves \eqref{heat},
\begin{align}\label{exten-2}
\int_{\mathbb{R}^n}|u(T,x)|^2 \,\mathrm dx \leq C(n,T)\left( \varepsilon\int_{\mathbb{R}^n}|u(0,x)|^2 \,\mathrm dx + b(\varepsilon) \int_{B_1}|u(T,x)|^2 \,\mathrm dx \right).
\end{align}
Let us  explain why  (\ref{exten-2}) deserves to be expected. {\it Reason One.} Let $\theta\in (0,1)$ and $T>0$. Then the next two inequalities are equivalent. The first inequality is as:
 there is $C(n, T,\theta)$ so that when $u$ solves \eqref{heat},
  \begin{align}\label{1.17GSwang}
\int_{\mathbb{R}^n}|u(T,x)|^2\,\mathrm dx\leq C(n, T,\theta)\Big( \int_{B_1}|u(T,x)|^2\,\mathrm dx\Big)^\theta\Big( \int_{\mathbb{R}^n}|u(0,x)|^2\,\mathrm dx\Big)^{1-\theta},
\end{align}
while the second inequality is as:  there is $C(n,T,\theta)>0$ so that for any
$\varepsilon\in(0,1)$ and any solution $u$ to (\ref{heat}),
\begin{align}\label{1.18GSwang}
\int_{\mathbb{R}^n}|u(T,x)|^2 \,\mathrm dx \leq C(n,T,\theta)\left( \varepsilon\int_{\mathbb{R}^n}|u(0,x)|^2 \,\mathrm dx + \varepsilon^{-\frac{1-\theta}{\theta}} \int_{B_1}|u(T,x)|^2 \,\mathrm dx \right).
\end{align}
 However, (\ref{1.18GSwang}) is not true, for otherwise, we can use the same method
 developed in \cite{PW} (see also \cite{AEWZ}) to derive (\ref{ob-q}) (where $E$ is replaced by $B_1$) which contradicts the equivalence of (i) and (iv) in  Theorem~\ref{equi-thm}. Thus, $b(\varepsilon)$ in (\ref{exten-2}) cannot grow like
 a polynomial of $\varepsilon$.
 But it seems not to be hopeless for us to find some kind of $b(\varepsilon)$ so that  (\ref{exten-2}) holds.
 {\it Reason Two.} The space-like strong unique continuation of the heat equation (\ref{heat}) (see \cite{EFV}) yields that if $u(T,\cdot)=0$ on the ball $B_1$, then $u(T,\cdot)= 0$ over $\mathbb R^n$. The inequality (\ref{exten-2}) is a quantitative version of the aforementioned  unique continuation.

Though we have not found any $b(\varepsilon)$ so that  (\ref{exten-2}) is true,
we obtained some $b(\varepsilon)$  so that
 (\ref{exten-2})  holds for all solutions to \eqref{heat} with
initial data having some slight decay (see Theorem \ref{app-obs} in Subsection \ref{07bao2}).

Finally, We would like to mention what follows: With the aid of an abstract lemma (i.e., Lemma 5.1 in \cite{WWZ}), each of extended inequalities mentioned above corresponds to a kind of controllability for the heat equation (\ref{heat}). We are not going to repeat the details on this issue in the current  paper.

 \end{description}

 \subsection{Plan of the paper}
 The paper is organized as follows: In Section 2, we prove Theorem~\ref{equi-thm}.
 In Section 3, we present several weak observability inequalities and weak interpolation inequalities, where observations are made in a ball of $\mathbb{R}^n$.

\section{Proof of Theorem \ref{equi-thm}}\label{dumain}

We are going to prove Theorem \ref{equi-thm} in the following way:
$$
{\bf (i) \Rightarrow (ii) \Rightarrow (iii) \Rightarrow (iv) \Rightarrow (i).}
 $$
The above steps are  based on  several lemmas: Lemmas \ref{lem-spec},
 \ref{lem-hold}, \ref{lem-obs} and \ref{lem-thick-E}.
 We begin with Lemma \ref{lem-spec} connecting  the spectral inequality with  sets of $\gamma$-thickness at scale $L$.

\begin{lemma}\label{lem-spec}
Suppose that  a measurable set $E\subset \mathbb{R}^n$ is $\gamma$-thick at scale $L$  for some  $\gamma>0$ and $L>0$. Then $E$ satisfies the spectral inequality (\ref{intro-spec}),
with
$$
C_{spec}(n,E) =C(1+L)\Big(1+\ln \frac{1}{\gamma}\Big)\;\;\mbox{for some}\;\;C=C(n).
$$
\end{lemma}

\begin{remark}\label{remark2.1wang}
The manner that the constant  $e^{C_{spec}(n,E)(1+N)}$ (in (\ref{intro-spec})) depends on $N$ is comparable with the manner that the constant $e^{C\sqrt \lambda}$ in \eqref{spec-bound}
 depends on $\lambda$. (This has been explained in the remark $(d_1)$ in Subsection 1.1.)
  The latter one  played an important role in the proof of
 the H\"older-type interpolation inequality (\ref{1.7GGWWSSG}) for the   heat equation (\ref{1.3WSGENG}) (see \cite{AEWZ}).
 Analogically, the previous one will play an important role in the proof of Theorem \ref{equi-thm}.

 \end{remark}

\begin{remark}\label{remark2.2wang}
In \cite{K}, the author announced the result in Lemma \ref{lem-spec} and proved it for the case when $n=1$.
  For the completeness of the paper, we give a detailed proof for Lemma \ref{lem-spec}, based on some ideas and techniques in  \cite{K}.
\end{remark}
To show Lemma \ref{lem-spec}, we need the following result  on analytic functions:
\begin{lemma}[{\cite[Lemma 1]{K}}]\label{lem-good}
Let $\Phi$ be an analytic function in $D_5(0)$ (the disc in $\mathbb{C}$, centered at origin and of radius 5). Let $I$ be an interval of length $1$ such that $0\in I$. Let $\hat E
\subset I$ be a subset of positive measure. If $|\Phi(0)|\geq 1$ and $M=\max_{|z|\leq 4}|\Phi(z)|$, then there exists a generic constant $C>0$ such that
$$
\sup_{x\in I}|\Phi(x)|\leq \left({C}/{|\hat E|}\right)^{\frac{\ln M}{\ln 2}}\sup_{x\in \hat E}|\Phi(x)|.
$$
\end{lemma}
We are now in the position to prove Lemma \ref{lem-spec}.

\begin{proof}[\textbf{Proof of Lemma \ref{lem-spec}}]
 We only need to prove this lemma for the case when  $L=1$. In fact,
 suppose that this is done. Let $E$ be $\gamma$-thick at scale $L>0$.
 Define a new set:
 $$
 L^{-1}E:=\big\{L^{-1}x: x\in E \big\}.
 $$
 One can easily check  that $L^{-1}E$ is $\gamma$-thick at scale $1$.
  Given
 $N>0$ and $f\in L^2(\mathbb{R}^n)$ with $supp \widehat{f} \subset B_N$, let
 $$
 g(x): =f(Lx),\; x\in\mathbb R^n.
 $$
 One can directly check that
 $$
 \widehat{g}(\xi)=L^{-n}\widehat{f}(L^{-1}\xi),\;\xi\in\mathbb R^n;\;\;\;supp\ \widehat{g}\subset B_{LN}.
 $$
  From these, we can apply Lemma \ref{lem-spec} (with $L=1$) to the set $L^{-1}E$ and the function $g$ to find
  $C=C(n)$ so that
 \begin{eqnarray}\label{wang2.1}
\int_{\mathbb{R}^n} |g(x)|^2\,\mathrm dx &\leq& e^{2C(1+\ln\frac{1}{\gamma})(1+LN)} \int_{L^{-1}E}|g(x)|^2\,\mathrm dx\nonumber\\
&\leq& e^{2C(1+L)(1+\ln\frac{1}{\gamma})(1+N)} \int_{L^{-1}E}|g(x)|^2\,\mathrm dx.
\end{eqnarray}
  Meanwhile, by changing
 variable $x\mapsto Lx$, we deduce that
\begin{eqnarray}\label{wang2.2}
\int_{\mathbb{R}^n} |f(x)|^2\,\mathrm dx =L^{-n}\int_{\mathbb{R}^n} |g(x)|^2\,\mathrm dx;\;\;
\int_{E}|f(x)|^2\,\mathrm dx=L^{-n}\int_{L^{-1}E}|g(x)|^2\,\mathrm dx.
\end{eqnarray}
Hence,  from (\ref{wang2.1}) and (\ref{wang2.2}), we find  that
$$
\int_{\mathbb{R}^n} |f(x)|^2\,\mathrm dx \leq e^{C_{spec}(n,E)(1+N)} \int_{E}|f(x)|^2\,\mathrm dx,
$$
with
$$
C_{spec}(n,E) =2C(1+L)\Big(1+\ln \frac{1}{\gamma}\Big).
$$
This proves the lemma  for  the general case that $L>0$.

We now show Lemma \ref{lem-spec} for the case when $L=1$ by several steps.
First of all, we arbitrarily fix $N>0$ and  $f\in L^2(\mathbb{R}^n)$ with  $supp \widehat{f} \subset B_N$. Without loss of generality, we can assume that $f\neq 0$.

{\it Step 1. Bad and good cubes.} For each multi-index $j=(j_1,j_2,\cdots, j_n)\in \mathbb{Z}^n$, let
$$
Q(j):=\left\{x=(x_1,\dots,x_n)\in \mathbb{R}^n: |x_i-j_i|< 1/2\;\;\mbox{for all}\;\; i=1,2,\cdots,n\right\}.
$$
It is clear that
$$
Q(j)\bigcap Q(k)={\O} \;\;\mbox{for all}\;\; j\neq k \in \mathbb{Z}^n;\;\;\mathbb{R}^n = \bigcup_{j\in \mathbb{Z}^n}\overline{Q(j)},
$$
where $\overline{Q(j)}$ denotes the closure of $Q(j)$. From these, we have that
\begin{align}\label{equ-spec-1}
\int_{\mathbb{R}^n}|f(x)|^2\,\mathrm dx = \sum_{j\in \mathbb{Z}^n}\int_{Q(j)}|f(x)|^2\,\mathrm dx.
\end{align}

We will divide $\{Q(j) : j\in \mathbb{Z}^n\}$ into two disjoint parts whose elements
are respectively called ``good cubes" and  ``bad cubes". And then we compare $\displaystyle\int_{\mathbb{R}^n}|f|^2$ with
$\displaystyle\int_{\bigcup{Q(j)\textmd{ is bad }}}|f|^2$
and $\displaystyle\int_{\bigcup{Q(j)\textmd{ is good }}}|f|^2$, respectively.
First, we define the function:
$$
h(s):=s^{n}(s-1)^{-n}-1,\; s\in [2,+\infty).
$$
It is a continuous and strictly decreasing function satisfying that
$$
h(2)\geq 1, \quad \lim_{s\rightarrow +\infty}h(s)=0.
$$
Thus we can take $A_0$ as the unique point in $[2,+\infty)$ so that $h(A_0)=1/2$. Clearly,
$A_0$ depends only on $n$, i.e., $A_0=A_0(n)$. Given $j\in \mathbb{Z}^n$, $Q(j)$ is said to be {\it a good cube}, if for each $\beta\in \mathbb{N}^n$,
\begin{align}\label{equ-spec-7}
\int_{Q(j)}|\partial^\beta_xf(x)|^2\,\mathrm dx \leq A_0^{|\beta|}N^{2|\beta|}\int_{Q(j)}|f(x)|^2\,\mathrm dx.
\end{align}
When $Q(j)$ is not a good cube, it is called as {\it a bad cube}. Thus, when $Q(j)$ is a bad cube, there is  $\beta \in \mathbb{N}^n$, with $|\beta|>0$,  so that
\begin{align}\label{equ-spec-4}
\int_{Q(j)}|\partial^\beta_xf(x)|^2\,\mathrm dx > A_0^{|\beta|}N^{2|\beta|}\int_{Q(j)}|f(x)|^2\,\mathrm dx.
\end{align}

Using the  Plancherel theorem and the assumption that $supp \widehat{f} \subset B_N(0)$, we obtain that for each $\beta\in \mathbb{N}^n$,
\begin{align}\label{equ-spec-3}
\int_{\mathbb{R}^n}|\partial^\beta_xf(x)|^2\,\mathrm dx&=\int_{\mathbb{R}^n}\Big|\widehat{\partial^\beta_xf}(\xi)\Big|^2\,\mathrm d\xi
 = \int_{\mathbb{R}^n} \Big|(i\xi)^\beta \widehat{f}(\xi)\Big|^2\,\mathrm d\xi= \int_{|\xi|
\leq N} \Big|\xi^\beta \widehat{f}(\xi)\Big|^2\,\mathrm d\xi\nonumber\\
&\leq N^{2|\beta|}\int_{|\xi|\leq N} |\widehat{f}(\xi)|^2\,\mathrm d\xi= N^{2|\beta|}\int_{\mathbb{R}^n}|f(x)|^2.
\end{align}
 Meanwhile, it follows by (\ref{equ-spec-4}) that when $Q(j)$ is a bad cube,
\begin{align}\label{equ-3-21-1}
\int_{Q(j)}|f(x)|^2\,\mathrm dx\leq \sum_{\beta \in \mathbb{N}^n, |\beta|>0}A_0^{-|\beta|}N^{-2|\beta|}\int_{Q(j)}|\partial^\beta_xf(x)|^2\,\mathrm dx.
\end{align}
Since $Q(j)$, $j\in \mathbb{Z}^n$, are disjoint, by taking the sum in \eqref{equ-3-21-1} for all bad cubes,  we find that
\begin{align}\label{equ-3-21-2}
\int_{\bigcup{Q(j)\textmd{ is bad }}}|f(x)|^2\,\mathrm dx &\leq \sum_{\beta \in \mathbb{N}^n, |\beta|>0}A_0^{-|\beta|}N^{-2|\beta|}\int_{\bigcup{Q(j)\textmd{ is bad }}}|\partial^\beta_xf(x)|^2\,\mathrm dx\nonumber\\
&\leq \sum_{\beta \in \mathbb{N}^n, |\beta|>0}A_0^{-|\beta|}N^{-2|\beta|}\int_{\mathbb{R}^n}|\partial^\beta_xf(x)|^2\,\mathrm dx.
\end{align}
From \eqref{equ-spec-3} and \eqref{equ-3-21-2}, we have that
\begin{align}\label{equ-spec-4.5}
\int_{\bigcup{Q(j)\textmd{ is bad }}}|f(x)|^2\,\mathrm dx
&\leq \sum_{\beta \in \mathbb{N}^n, |\beta|>0}A_0^{-|\beta|}\int_{\mathbb{R}^n}|f(x)|^2\,\mathrm dx\nonumber\\
&=\big(A_0^{n}(A_0-1)^{-n}-1\big)\int_{\mathbb{R}^n}|f(x)|^2\,\mathrm dx.
\end{align}
Since $h(A_0)=\frac{1}{2}$, it follows from (\ref{equ-spec-4.5}) that
\begin{align}\label{equ-spec-5}
\int_{\bigcup{Q(j)\textmd{ is bad }}}|f(x)|^2\,\mathrm dx \leq \frac{1}{2} \int_{\mathbb{R}^n}|f(x)|^2\,\mathrm dx.
\end{align}
By \eqref{equ-spec-1} and \eqref{equ-spec-5}, we obtain that
\begin{align}\label{equ-spec-6}
\int_{\bigcup{Q(j)\textmd{ is good }}}|f(x)|^2\,\mathrm dx \geq \frac{1}{2}\int_{\mathbb{R}^n}|f(x)|^2\,\mathrm dx.
\end{align}

{\it Step 2.  Properties on good cubes.}
Arbitrarily fix  a good cube ${Q(j)}$.
We will prove some properties related to ${Q(j)}$.
First of all, we  claim that there is $C_0(n)>0$ so that
\begin{align}\label{good-7}
\left\|\partial^\beta_xf\right\|_{L^\infty(Q(j))}\leq C_0(n)(1+N)^n\left(\sqrt{A_0}N\right)^{|\beta|}\|f\|_{L^2(Q(j))}\;\;\mbox{for all}\;\;
\beta\in \mathbb{N}^n.
\end{align}
In fact, according to \eqref{equ-spec-7},  there is $C_1(n)>0$ so that
\begin{align}\label{good-6}
\|\partial^\beta_xf\|_{W^{n,2}(Q(j))}&=\sum_{\mu\in \mathbb{N}^n, |\mu|\leq n}\left(\int_{Q(j)}|\partial^{\beta+\mu}_xf(x)|^2\,\mathrm dx\right)^{1/2}\nonumber\\
&\leq \sum_{\mu\in \mathbb{N}^n, |\mu|\leq n} A_0^{\frac{|\beta+\mu|}{2}}N^{|\beta+\mu|}\|f\|_{L^2(Q(j))}\nonumber\\
&\leq C_1(n)(1+N)^n\left(\sqrt{A_0}N\right)^{|\beta|}\|f\|_{L^2(Q(j))}\;\;\mbox{for all}\;\;
\beta\in \mathbb{N}^n.
\end{align}
Meanwhile, because $Q(j)$ satisfies the cone condition, we can apply the Sobolev embedding theorem $W^{n,2}(Q(j))\hookrightarrow L^\infty(Q(j))$ to find $C_2(n)>0$ so that
$$
\|\varphi\|_{L^\infty(Q(j))}\leq C_2(n)\|\varphi\|_{W^{n,2}(Q(j))}\;\; \mbox{for all}\;\; \varphi\in W^{n,2}(Q(j)).
$$
 This, along with   \eqref{good-6}, leads to (\ref{good-7}).

Next, we let $y\in\overline{Q(j)}$ satisfy that
\begin{align}\label{good-1}
\|f\|_{L^\infty(Q(j))}=|f(y)|.
\end{align}
(Due to the continuity of $|f|$ over $\mathbb{R}^n$, such $y$ exists.)
 Because the diameter of $Q(j)$ is $\sqrt{n}$, we can use the spherical coordinates centered at $y$ to obtain that
 \begin{align}\label{good-2}
| E\bigcap Q(j)|&=\int_0^\infty dr\int_{|x-y|=r}\chi_{E\bigcap Q(j)}(x)d\sigma\nonumber\\
&=\int_0^{\sqrt{n}} dr\int_{|x-y|=r}\chi_{E\bigcap Q(j)}(x)d\sigma\nonumber\\
&=\sqrt{n}\int_0^1 dr\int_{|x-y|=\sqrt{n}r}\chi_{E\bigcap Q(j)}(x)d\sigma.
\end{align}
In \eqref{good-2}, we change the variable:
 $$
 x=y+\sqrt{n}rw\;\;\mbox{ with}\;\;w\in \mathbb{S}^{n-1},
 $$
  and then obtain
that
\begin{align}\label{good-3}
| E\bigcap Q(j)|&=\sqrt{n}\int_0^1 (\sqrt{n} r)^{n-1}dr\int_{\mathbb{S}^{n-1}}\chi_{E\bigcap Q(j)}(y+\sqrt{n}rw)d\sigma\nonumber\\
&\leq n^{n/2}\int_0^1 dr\int_{\mathbb{S}^{n-1}}\chi_{E\bigcap Q(j)}(y+\sqrt{n}rw)d\sigma.
\end{align}
For each $w\in \mathbb{S}^{n-1}$, let
\begin{eqnarray}\label{2.17GGSSWang}
I_{w}\triangleq \Big\{r\in [0,1]\;:\; y+\sqrt{n}rw\in E\bigcap Q(j)\Big\}.
\end{eqnarray}
Since
$$
|\mathbb{S}^{n-1}|=\frac{2\pi^{\frac{n}{2}}}{\Gamma(\frac{n}{2})},\;\;
\mbox{where}\;\;\Gamma(\cdot)\;\;\mbox{is the Gamma function},
$$
it follows from \eqref{equ-set} and \eqref{good-3} that
\begin{align}\label{good-4}
|I_{\omega_0}|\geq \frac{|E\bigcap Q(j)|}{n^{n/2}|\mathbb{S}^{n-1}|}\geq \frac{\Gamma(\frac{n}{2})}{2(n\pi)^{n/2}}\gamma\;\;\mbox{for some}\;\;w_0\in \mathbb{S}^{n-1}.
\end{align}

Then  we define a function $\phi(\cdot)$ over $[0,1]$ by
\begin{equation}\label{dub2}
\phi(t)=\frac{f\left(y+\sqrt{n}tw_0\right)}{\|f\|_{L^2(Q(j))}},\;\; t\in [0,1].
\end{equation}
(Since $f\in L^2(\mathbb{R}^n)$ satisfies that  $supp \;\widehat{f} \subset B_N$, we have that $f$ is analytic over $\mathbb{R}^n$. Consequently, $\|f\|_{L^2(Q(j))}\neq 0$ because we assumed that $f\neq 0$ over $\mathbb{R}^n$.)

We claim that $\phi(t)$ can be extended to  an entire function in the complex plane.
In fact, by (\ref{dub2}), one can directly check that
\begin{align}\label{good-5}
\big|\phi^{(k)}(0)\big|\leq \frac{n^{\frac{3}{2}k}\max_{|\beta|=k}\big\|D^\beta f\big\|_{L^\infty(Q(j))}}{\|f\|_{L^2(Q(j))}}\;\;
\mbox{for all}\;\; k\geq0.
\end{align}
By \eqref{good-5} and \eqref{good-7}, we see that
\begin{align}\label{good-8}
\big|\phi^{(k)}(0)\big|\leq C_0(n)(1+N)^n\Big(n^{\frac{3}{2}}\sqrt{A_0}N\Big)^{k}\;\;
\mbox{for all}\;\; k\geq0.
\end{align}
From (\ref{good-8}), we find that
 \begin{align}\label{good-4.5}
\phi(t)=\phi(0)+\phi'(0)t+\cdots+\frac{\phi^{(k)}(0)}{k!}t^k+\cdots, \quad t\in [0,1],
\end{align}
and that the series in (\ref{good-4.5}), with $t$ being replaced by any $z\in \mathbb{C}$, is convergent.  Thus, the above claim is true. {\it From now on, we will use $\phi(z)$ to denote the extension of $\phi(t)$ over $\mathbb{C}$.}

{\it Step 3.  Recovery of the $L^2(\mathbb R^n)$ norm.}
We will finish our proof in this step.
Applying Lemma \ref{lem-good}, where
$$
I=[0,1],\;\;\hat E=I_{w_0}\;(\mbox{defined by}\;(\ref{2.17GGSSWang})\;\mbox{and}\; (\ref{good-4}))\; \;\mbox{and}\;\;\Phi=\phi,
$$
 and then using \eqref{good-4},  we can find $C_3=C_3(n)$ so that
\begin{align}\label{good-9}
\sup_{t\in [0,1]}|\phi(t)|&\leq \left({C}/{|I_{w_0}|}\right)^{\frac{\ln M}{\ln 2}}\sup_{t\in I_{w_0}}|\phi(t)|\nonumber\\
&\leq \left(\frac{2 C(n\pi)^{n/2}}{\gamma\Gamma(\frac{n}{2})}\right)^{\frac{\ln M}{\ln 2}}\sup_{t\in I_{w_0}}|\phi(t)|\nonumber\\
&\leq M^{C_3\left(1+\ln\frac{1}{\gamma}\right)} \sup_{t\in I_{w_0}}|\phi(t)|,
\end{align}
 where
\begin{align}\label{dgood-16}
M &= \max_{|z|\leq 4}|\phi(z)|.
\end{align}
Two facts are given in order. First, it follows from \eqref{good-1} and \eqref{dub2} that
$$
\frac{\|f\|_{L^\infty(Q(j))}}{\|f\|_{L^2(Q(j))}}
=\frac{|f(y)|}{\|f\|_{L^2(Q(j))}}=|\phi(0)|.
$$
Second, it follows  by the definition of $I_{w_0}$ (see (\ref{2.17GGSSWang})
and
(\ref{good-4})) that
$$
\sup_{t\in I_{w_0}}|\phi(t)|\leq \frac{\|f\|_{L^\infty(E\bigcap Q(j))}}{\|f\|_{L^2(Q(j))}}.
$$
The above two facts, along with  \eqref{good-9}, yield  that
\begin{align}\label{good-10}
\|f\|_{L^\infty(Q(j))}\leq M^{C_3\left(1+\ln\frac{1}{\gamma}\right)} \|f\|_{L^\infty(E\bigcap Q(j))}.
\end{align}

We next define
$$
E'\triangleq \Big\{x\in E\bigcap Q(j): |f(x)|\leq \frac{2}{|E\bigcap Q(j)|}\int_{E\bigcap Q(j)}|f(x)|\,\mathrm dx\Big\}.
$$
By the Chebyshev inequality, we have that
$$
|E'|\geq \frac{|E\bigcap Q(j)|}{2}\geq\frac{\gamma}{2}.
$$
By the same argument as that used in the proof of \eqref{good-10},  one can obtain that
\begin{align}\label{good-11}
\|f\|_{L^\infty(Q(j))}\leq M^{C_4\left(1+\ln\frac{1}{\gamma}\right)} \|f\|_{L^\infty(E'\bigcap Q(j))}
\;\;\mbox{for some}\;\; C_4=C_4(n).
\end{align}
Meanwhile, it follows by the definition of $E'$ that
\begin{align}\label{good-12}
 \|f\|_{L^\infty(E'\bigcap Q(j))}\leq \frac{2}{|E\bigcap Q(j)|}\int_{E\bigcap Q(j)}|f(x)|\,\mathrm dx.
\end{align}
From \eqref{good-11}, \eqref{good-12} and  the H\"{o}lder inequality,
we find
that
\begin{align}\label{good-15}
\int_{Q(j)}|f|^2\,\mathrm dx \leq \frac{4}{\gamma}M^{2C_4\left(1+\ln\frac{1}{\gamma}\right)}\int_{E\bigcap Q(j)}|f|^2\,\mathrm dx.
\end{align}
The term $M$ (given by \eqref{dgood-16}) can be estimated by
 \eqref{good-8} as follows:
\begin{align}\label{good-16}
M &= \max_{|z|\leq 4}|\phi(z)|\nonumber\\
&\leq \max_{|z|\leq 4}\sum_{k=0}^\infty\frac{\phi^{(k)}(0)}{k!}|z|^k\nonumber\\
&\leq C_2(n)(1+N)^n\sum_{k=0}^\infty \frac{\big(4n^{\frac{3}{2}}\sqrt{A_0}N\big)^{k}}{k!}\nonumber\\
&\leq e^{C_5(1+N)}\;\;\mbox{for some}\;\; C_5=C_5(n).
\end{align}

Finally, combining \eqref{good-15} and \eqref{good-16} leads to that
\begin{align}\label{good-17}
\int_{Q(j)}|f|^2\,\mathrm dx \leq e^{C_6(1+N)\left(1+\ln\frac{1}{\gamma}\right)}\int_{E\bigcap Q(j)}|f|^2\,\mathrm dx \;\;\mbox{for some}\;\; C_6=C_6(n).
\end{align}
Taking the sum in \eqref{good-17} for all good cubes, using \eqref{equ-spec-6},   we
see that
\begin{align*}
\int_{\mathbb{R}^n}|f|^2\,\mathrm dx &\leq 2\sum_{ Q(j)\textmd{ is good }}\int_{Q(j)}|f|^2\,\mathrm dx\\
 &\leq 2\sum_{ Q(j)\textmd{ is good }}e^{C_6(1+N)\left(1+\ln\frac{1}{\gamma}\right)}\int_{E\bigcap Q(j)}|f|^2\,\mathrm dx\\
&\leq 2e^{C_6(1+N)\left(1+\ln\frac{1}{\gamma}\right)}\int_{E}|f|^2\,\mathrm dx,
\end{align*}
which leads to (\ref{intro-spec}),
where
$$
C_{spec}(n,E) =(C_6+1)(1+L)\Big(1+\ln \frac{1}{\gamma}\Big),\;\;\mbox{with}\;\;L=1.
$$
This ends the proof of Lemma \ref{lem-spec}.
\end{proof}

 Lemma \ref{lem-hold} and Lemma \ref{lem-obs}
deal with  connections among the spectral inequality (\ref{intro-spec}), the H\"older-type interpolation
inequality \eqref{interpolation}
and the observability inequality (\ref{wangobs1.8}). In their proofs, we  borrowed  some ideas and techniques from  \cite{AEWZ, PW}.

\begin{lemma}\label{lem-hold}
Suppose that a measurable set $E\subset \mathbb{R}^n$ satisfies the spectral inequality
(\ref{intro-spec}). Then $E$ satisfies
the H\"older-type interpolation \eqref{interpolation}, with
$$
C_{Hold} = \frac{1}{1-\theta}(C_{spec}+1)^2 + \ln 12.
$$
\end{lemma}

\begin{proof}
Let $E\subset \mathbb{R}^n$ satisfy the spectral inequality
(\ref{intro-spec}). Arbitrarily fix $T>0$, $\theta\in (0,1)$ and
a solution $u$ to \eqref{heat}. Write
$$
u_0(x)=u(0,x),\; x\in \mathbb{R}^n.
$$
 Then
we have that
$$
u(T,x)=\left(e^{T\triangle}u_0\right)(x)\;\;\mbox{ for all}\;\;x\in \mathbb{R}^n.
$$
 Given $N>0$, write respectively
$\chi_{\leq N}(D)$ and $\chi_{> N}(D)$ for the  multiplier operators with the symbols $\chi_{\{|\xi|\leq N\}}$ and $\chi_{\{|\xi|> N\}}$. Namely, for each $ g \in L^2(\mathbb{R}^n)$,
$$
\widehat{\chi_{\leq N}(D)g}(\xi) = \chi_{\{|\xi|\leq N\}}\widehat{g}(\xi)
\;\;\mbox{and}\;\;\widehat{\chi_{> N}(D) g}(\xi) = \chi_{\{|\xi|>N\}}\widehat{g}(\xi)
\;\;\mbox{for a.e.}\;\; \xi\in\mathbb{R}^n.
$$
Then we can express    $u_0$ as:
$$
u_0 = \chi_{\leq N}(D)u_0 + \chi_{> N}(D)u_0.
$$
From this and (\ref{intro-spec}), we can easily check that
\begin{align}\label{wang2.30}
&\int_{\mathbb{R}^n}|u(T,x)|^2\,\mathrm dx\nonumber\\
&\leq 2
\int_{\mathbb{R}^n}\big|\big(e^{T\triangle}\chi_{\leq N}(D)u_0\big)(x)\big|^2\,\mathrm dx + 2\int_{\mathbb{R}^n}\big|\big(e^{T\triangle}\chi_{> N}(D)u_0\big)(x)\big|^2\,\mathrm dx\nonumber\\
&\leq 2e^{C_{spec}(1+N)}\int_{E}\big|\big(e^{T\triangle}\chi_{\leq N}(D)u_0\big)(x)\big|^2\,\mathrm dx + 2\int_{\mathbb{R}^n}\big|\big(e^{T\triangle}\chi_{> N}(D)u_0\big)(x)\big|^2\,\mathrm dx\nonumber\\
&\leq 4e^{C_{spec}(1+N)}\int_{E}\big|\big(e^{T\triangle}u_0\big)(x)\big|^2\,\mathrm dx+ \Big(2+4e^{C_{spec}(1+N)}\Big)\int_{\mathbb{R}^n}\big|\big(e^{T\triangle}\chi_{> N}(D)u_0\big)(x)\big|^2\,\mathrm dx.
\end{align}
Since
\begin{align*}
\int_{\mathbb{R}^n}\big|\big(e^{T\triangle}\chi_{> N}(D)u_0\big)(x)\big|^2\,\mathrm dx &= \int_{\mathbb{R}^n}\big|e^{-T|\xi|^2}\chi_{> N}(\xi)\widehat{u_0}(\xi)\big|^2\,\mathrm d\xi\\
&\leq e^{-TN^2}\int_{\mathbb{R}^n}\big|\widehat{u_0}(\xi)\big|^2\,\mathrm d\xi\\
 &= e^{-TN^2}\int_{\mathbb{R}^n}|u_0(x)|^2\,\mathrm dx,
\end{align*}
it follows from (\ref{wang2.30}) that
\begin{align}\label{wang2.31}
\int_{\mathbb{R}^n}|u(T,x)|^2\,\mathrm dx
&\leq 4e^{C_{spec}(1+N)}\int_{E}|u(T,x)|^2\,\mathrm dx + \big(2+4e^{C_{spec}(1+N)}\big)e^{-TN^2}\int_{\mathbb{R}^n}|u_0(x)|^2\,\mathrm dx\nonumber\\
&\leq 6e^{C_{spec}}\left(e^{C_{spec}N}\int_{E}|u(T,x)|^2\,\mathrm dx + e^{C_{spec}N-TN^2}\int_{\mathbb{R}^n}|u_0(x)|^2\,\mathrm dx\right).
\end{align}

Given $\varepsilon \in (0,1)$, choose $N=N(\varepsilon)$ so that
$$
 \exp\left[{C_{spec}N-TN^2}\right]=\varepsilon.
$$
(This can be done since the set: $\{C_{spec}s-Ts^2: s>0\}$ contains $(-\infty,0]$.)
 With the above choice of $N$, we have that
$$
N=\frac{C_{spec}+\sqrt{ C^2_{spec}+4T\ln\frac{1}{\varepsilon}}}{2T}\leq \frac{1}{T}\left(C_{spec}+\sqrt{T\ln\frac{1}{\varepsilon}}\right).
$$
Thus, with $\theta\in (0,1)$ fixed before, we see that
\begin{align*}
\exp\left[C_{spec}N\right]&\leq \exp\left[\frac{C^2_{spec}}{T}\right]\exp\left[\frac{C_{spec}}
{\sqrt{T}}\sqrt{\ln\frac{1}{\varepsilon}}\right]\\
&\leq \exp\left[\frac{C^2_{spec}}{T}\right]
\exp\left[\frac{1-\theta}{\theta}\ln\frac{1}{\varepsilon}
+\frac{\theta}{1-\theta}\frac{C^2_{spec}}{T}\right]\\
&= \exp\left[{\frac{ C^2_{spec}}{(1-\theta)T}}\right]\varepsilon^{-\frac{1-\theta}{\theta}}.
\end{align*}
From this and (\ref{wang2.31}), we find that  for every $\varepsilon\in (0,1)$,
\begin{align*}
\int_{\mathbb{R}^n}|u(T,x)|^2\,\mathrm dx\leq 6e^{C_{spec}}\left(e^{\frac{ C^2_{spec}}{(1-\theta)T}}\varepsilon^{-\frac{1-\theta}{\theta}}\int_{E}|u(T,x)|^2\,\mathrm dx + \varepsilon\int_{\mathbb{R}^n}|u_0(x)|^2\,\mathrm dx\right).
\end{align*}
Choosing in the above
$$
\varepsilon = \left( \frac{\int_{E}|u(T,x)|^2\,\mathrm dx}{\int_{\mathbb{R}^n}|u_0(x)|^2\,\mathrm dx}\right)^\theta,
$$
we obtain that
\begin{align*}
\int_{\mathbb{R}^n}|u(T,x)|^2\,\mathrm dx&\leq 12e^{C_{spec}}e^{\frac{ C^2_{spec}}{(1-\theta)T}}\left( \int_{E}|u(T,x)|^2\,\mathrm dx\right)^\theta\left( \int_{\mathbb{R}^n}|u_0(x)|^2\,\mathrm dx\right)^{1-\theta}\\
&\leq  e^{\left(\frac{1}{1-\theta}(C_{spec}+1)^2 + \ln 12\right)\left(1+\frac{1}{T}\right)}\left( \int_{E}|u(T,x)|^2\,\mathrm dx\right)^\theta\left( \int_{\mathbb{R}^n}|u_0(x)|^2\,\mathrm dx\right)^{1-\theta},
\end{align*}
which leads to \eqref{interpolation} with
$$
C_{Hold} = \frac{1}{1-\theta}\left(C_{spec}+1\right)^2 + \ln 12.
$$
This ends the proof of Lemma \ref{lem-hold}.
\end{proof}

\begin{lemma}\label{lem-obs}
Suppose that a measurable set  $E\subset \mathbb{R}^n$ has the property:
there is a positive constant $C_{Hold}=C_{Hold}(n,E)$ so that  for any $T>0$,
\begin{align}\label{assume-interpolation}
\int_{\mathbb{R}^n}|u(T,x)|^2\,\mathrm dx\leq e^{C_{Hold}\big(1+\frac{1}{T}\big)}\Big( \int_{E}|u(T,x)|^2\,\mathrm dx\Big)^{1/2}\Big( \int_{\mathbb{R}^n}|u(0,x)|^2\,\mathrm dx\Big)^{1/2},
\end{align}
when $u$ solves the equation \eqref{heat}.
Then  for each $T>0$ and each subset $F\subset (0,T)$  of positive measure, there is a positive constant $C_{obs}=C_{obs}(n,T,F,C_{Hold})$   so that when $u$ solves \eqref{heat},
\begin{equation}\label{duc2}
\int_{\mathbb{R}^n}|u(T,x)|^2\,\mathrm dx\leq C_{obs} \int_{F} \int_{E}|u(s,x)|^2\,\mathrm dx\mathrm ds.
\end{equation}
In particular, if $F=(0,T)$ then the constant $C_{obs}$ in (\ref{duc2}) can be expressed as:
$$
C_{obs}=\exp\left[36(1+3C_{Hold})(1+{1}/{T})\right].
$$
\end{lemma}

\begin{proof}
Suppose that $E\subset \mathbb{R}^n$ satisfies (\ref{assume-interpolation}).
Arbitrarily fix $T>0$ and $F\subset (0,T)$ of positive measure.
Applying Cauchy's inequality to \eqref{assume-interpolation}, we find that for all $t>0$ and $\varepsilon>0$,
\begin{align}\label{prod-1}
\int_{\mathbb{R}^n}|u(t,x)|^2\,\mathrm dx &\leq \frac{1}{\varepsilon}e^{2C_{Hold}\left(1+\frac{1}{t}\right)} \int_{E}|u(t,x)|^2\,\mathrm dx+\varepsilon \int_{\mathbb{R}^n}|u_0(x)|^2\,\mathrm dx.
\end{align}
By a translation in time, we find from (\ref{prod-1}) that  for all $0<t_1<t_2$ and $\varepsilon>0$,
\begin{align}\label{prod-2}
\int_{\mathbb{R}^n}|u(t_2,x)|^2\,\mathrm dx \leq \frac{1}{\varepsilon}e^{2C_{Hold}\left(1+\frac{1}{t_2-t_1}\right)} \int_{E}|u(t_2,x)|^2\,\mathrm dx+\varepsilon \int_{\mathbb{R}^n}|u(t_1,x)|^2\,\mathrm dx.
\end{align}

Let $l$ be a Lebesgue density point of $F$. Then according to \cite[Proposition 2.1]{PW},
for each $\lambda\in(0,1)$,  there is a sequence $\{l_m\}_{l=1}^\infty\subset (l,T)$ so that for each $m \in \mathbb{N}^+$,
\begin{align}\label{prod-3}
l_{m+1}-l=\lambda^m(l_1-l)
\end{align}
and
\begin{align}\label{prod-4}
\Big|F\bigcap (l_{m+1},l_m)\Big|\geq \frac{1}{3}(l_m-l_{m+1}).
\end{align}
Arbitrarily fix  $m\in\mathbb{N}^+$. Take $s$  so that
$$
0<l_{m+2}<l_{m+1}\leq s<l_m<T.
 $$
 Using \eqref{prod-2} (with $t_1=l_{m+2}$ and $t_2=s$) and noting that
$$
\int_{\mathbb{R}^n}\left|u\left(l_m,x\right)\right|^2\,\mathrm dx \leq \int_{\mathbb{R}^n}|u(s,x)|^2\,\mathrm dx
\;\;\mbox{and}\;\;l_{m+1}-l_{m+2}\leq s-l_{m+2},
$$
 we see that
\begin{align}\label{prod-5}
\int_{\mathbb{R}^n}\left|u\left(l_m,x\right)\right|^2\,\mathrm dx \leq \frac{1}{\varepsilon}e^{2C_{Hold}\left(1+\frac{1}{l_{m+1}-l_{m+2}}\right)} \int_{E}|u(s,x)|^2\,\mathrm dx+\varepsilon \int_{\mathbb{R}^n}\left|u\left(l_{m+2},x\right)\right|^2\,\mathrm dx.
\end{align}
Integrating with $s$ over $F\bigcap (l_{m+1},l_m)$ in \eqref{prod-5} implies that
\begin{multline}\label{prod-6}
\int_{\mathbb{R}^n}\left|u\left(l_m,x\right)\right|^2\,\mathrm dx \leq \varepsilon \int_{\mathbb{R}^n}\left|u\left(l_{m+2},x\right)\right|^2\,\mathrm dx \\
+\frac{1}{\varepsilon} \frac{1}{\left|F\bigcap (l_{m+1},l_m)\right|}e^{2C_{Hold}\left(1+\frac{1}{l_{m+1}-l_{m+2}}\right)}\int_{F\bigcap (l_{m+1},l_m)} \int_{E}|u(s,x)|^2\,\mathrm dx \mathrm ds.
\end{multline}
Since it follows by \eqref{prod-4} that
$$
\left|F\bigcap \left(l_{m+1},l_m\right)\right|\geq \frac{1}{3}\left(l_m-l_{m+1}\right)\geq \frac{1}{3}e^{-\frac{1}{l_m-l_{m+1}}},
$$
we obtain from  \eqref{prod-6} that
\begin{eqnarray}\label{prod-7}
\int_{\mathbb{R}^n}\left|u\left(l_m,x\right)\right|^2\,\mathrm dx &\leq& \varepsilon \int_{\mathbb{R}^n}\left|u\left(l_{m+2},x\right)\right|^2\,\mathrm dx \\
&+&\frac{3}{\varepsilon} e^{\frac{1}{l_m-l_{m+1}}+2C_{Hold}\left(1+\frac{1}{l_{m+1}-l_{m+2}}\right)}\int_{F\bigcap \left(l_{m+1},l_m\right)} \int_{E}|u(s,x)|^2\,\mathrm dx \mathrm ds.\nonumber
\end{eqnarray}
Meanwhile, it follows by \eqref{prod-3}  that
\begin{align}\label{prod-8}
l_m-l_{m+1}=\frac{1}{1+\lambda}\left(l_m-l_{m+2}\right)
\end{align}
and
\begin{align}\label{prod-9}
l_{m+1}-l_{m+2}=\frac{\lambda}{1+\lambda}\left(l_m-l_{m+2}\right).
\end{align}
Inserting  \eqref{prod-8} and \eqref{prod-9} into \eqref{prod-7}, we find that
\begin{eqnarray}\label{prod-10}
\int_{\mathbb{R}^n}\left|u\left(l_m,x\right)\right|^2\,\mathrm dx &\leq& \varepsilon \int_{\mathbb{R}^n}\left|u\left(l_{m+2},x\right)\right|^2\,\mathrm dx\nonumber\\
&+&3e^{2C_{Hold}}\frac{1}{\varepsilon} e^{\frac{C'}{l_m-l_{m+2}}}\int_{F\bigcap \left(l_{m+1},l_m\right)} \int_{E}|u(s,x)|^2\,\mathrm dx \mathrm ds
\end{eqnarray}
with
\begin{eqnarray}\label{wang2.44}
C'=1+\lambda+\frac{2C_{Hold}(1+\lambda)}{\lambda}.
\end{eqnarray}
 Rewrite \eqref{prod-10} as
\begin{eqnarray}\label{prod-11}
&&\varepsilon e^{-\frac{C'}{l_m-l_{m+2}}}\int_{\mathbb{R}^n}\left|u\left(l_m,x\right)\right|^2\,\mathrm dx - \varepsilon^2e^{-\frac{C'}{l_m-l_{m+2}}} \int_{\mathbb{R}^n}\left|u\left(l_{m+2},x\right)\right|^2\,\mathrm dx\nonumber\\
&\leq& 3e^{2C_{Hold}} \int_{F\bigcap \left(l_{m+1},l_m\right)} \int_{E}|u(s,x)|^2\,\mathrm dx \mathrm ds.
\end{eqnarray}

Next, we fix $\lambda \in (1/\sqrt{2},1)$. Let  $\mu:=\frac{1}{2-\lambda^{-2}}$.
 Then $\mu>1$.
  Setting, in \eqref{prod-11},
   $$
   \varepsilon=\exp\left[-\frac{(\mu-1)C'}{l_m-l_{m+2}}\right],
   $$
    we have that
\begin{eqnarray}\label{equ-inter-10-21-1}
&&e^{-\frac{\mu C'}{l_m-l_{m+2}}}\int_{\mathbb{R}^n}\left|u\left(l_m,x\right)\right|^2\,\mathrm dx - e^{-\frac{(2\mu-1)C'}{l_m-l_{m+2}}} \int_{\mathbb{R}^n}\left|u\left(l_{m+2},x\right)\right|^2\,\mathrm dx\nonumber\\
&\leq& 3e^{2C_{Hold}} \int_{F\bigcap \left(l_{m+1},l_m\right)} \int_{E}|u(s,x)|^2\,\mathrm dx \mathrm ds.
\end{eqnarray}
Meanwhile, one can easily check  that
\begin{equation}\label{equ-inter-10-21-2}
\exp\left[-\frac{(2\mu-1)C'}{l_m-l_{m+2}}\right] = \exp\left[-\frac{\mu C'}{\lambda^2(l_m-l_{m+2})}\right].
\end{equation}
 Because
 $$
 l_{m+2}-l_{m+4}=\lambda^2\left(l_m-l_{m+2}\right),
  $$
  we deduce from \eqref{equ-inter-10-21-1} and \eqref{equ-inter-10-21-2} that
\begin{eqnarray*}
&&e^{-\frac{\mu C'}{l_m-l_{m+2}}}\int_{\mathbb{R}^n}\left|u\left(l_m,x\right)\right|^2\,\mathrm dx - e^{-\frac{\mu C'}{l_{m+2}-l_{m+4}}} \int_{\mathbb{R}^n}\left|u\left(l_{m+2},x\right)\right|^2\,\mathrm dx\\
&\leq& 3e^{2C_{Hold}} \int_{F\bigcap \left(l_{m+1},l_m\right)} \int_{E}|u(s,x)|^2\,\mathrm dx \mathrm ds.
\end{eqnarray*}
Summing  the above inequality for all odd $m$ derives that
\begin{align*}
e^{-\frac{\mu C'}{l_1-l_{3}}}\int_{\mathbb{R}^n}\left|u\left(l_1,x\right)\right|^2\,\mathrm dx
&\leq 3e^{2C_{Hold}} \sum_{m=1}^\infty\int_{F\bigcap \left(l_{m+1},l_m\right)} \int_{E}|u(s,x)|^2\,\mathrm dx \mathrm ds\\
&\leq 3e^{2C_{Hold}} \int_{F\bigcap \left(l,l_1\right)} \int_{E}|u(s,x)|^2\,\mathrm dx \mathrm ds\nonumber\\
&\leq 3e^{2C_{Hold}} \int_{F} \int_{E}|u(s,x)|^2\,\mathrm dx \mathrm ds.
\end{align*}
Thus, we have that
\begin{align}\label{equ-obs-10-21}
\int_{\mathbb{R}^n}|u(T,x)|^2\,\mathrm dx\leq \int_{\mathbb{R}^n}|u(l_1,x)|^2\,\mathrm dx \leq 3e^{2C_{Hold}}e^{\frac{\mu C'}{l_1-l_{3}}} \int_{F} \int_{E}|u(s,x)|^2\,\mathrm dx \mathrm ds,
\end{align}
which leads to \eqref{duc2} with
$$
C_{obs}=3\exp\left[2C_{Hold}+{\frac{\mu C'}{l_1-l_{3}}}\right].
$$

Finally, in the case when  $F=(0,T)$, we set
$$
l_1=\frac{2T}{3},\;\; l=\frac{T}{3}\;\;\mbox{ and}\;\;\lambda=\sqrt{\frac{2}{3}}.
  $$
  Then we have that (see (\ref{wang2.44}))
  $$
 l_1-l_3=\frac{T}{9},\; \mu=2\;\;\mbox{ and}\;\;C'\leq 2+6C_{Hold}.
    $$
    Now, we derive from \eqref{equ-obs-10-21} that
\begin{align*}
\int_{\mathbb{R}^n}|u(T,x)|^2\,\mathrm dx &\leq 3e^{2C_{Hold}}e^{\frac{36\left(1+3C_{Hold}\right)}{T}} \int_{0}^T \int_{E}|u(s,x)|^2\,\mathrm dx \mathrm ds\\
&\leq  e^{36\left(1+3C_{Hold}\right)\left(1+\frac{1}{T}\right)} \int_{0}^T \int_{E}|u(s,x)|^2\,\mathrm dx \mathrm ds.
\end{align*}
This completes the proof of Lemma \ref{lem-obs}.
\end{proof}

  The  next lemma seems to be new to our best knowledge. The key of its proof is  the structure of a special solution to the equation (\ref{heat}). This structure is based on  the heat kernel.
\begin{lemma}\label{lem-thick-E}
Suppose that a measurable set  $E\subset\mathbb{R}^n$
satisfies the observability inequality (\ref{ob-q}).
Then the set $E$ is
 $\gamma$-thick at scale $L$ for some $\gamma>0$ and $L>0$.
\end{lemma}

\begin{proof}
Let   $E\subset\mathbb{R}^n$ be a measurable set
satisfying the observability inequality (\ref{ob-q}). Recall that the heat kernel is as:
$$
K(t,x)=(4\pi t)^{-n/2}e^{-|x|^2/4t},\;\; t>0, x\in\mathbb R^n.
$$
Given $u_0\in L^2(\mathbb R^n)$, the function defined by
\begin{eqnarray}\label{2.49Wang}
(t,x) \longrightarrow\int_{\mathbb{R}^n}K(t,x-y)u_0(y)\,\mathrm dy,\;\; (t,x)\in(0,\infty)\times\mathbb R^n,
\end{eqnarray}
is a solution to the equation \eqref{heat} with the initial condition $u(0,x)=u_0(x)$, $x\in \mathbb{R}^n$.

Arbitrarily fix $x_0\in \mathbb{R}^n$. By taking
$$
u_0(x)=(4\pi )^{-n/2}e^{-|x-x_0|^2/4},\;\;x\in\mathbb R^n,
$$
in (\ref{2.49Wang}), we get the following solution to the equation \eqref{heat}:
\begin{align}\label{special-solution}
v(t,x)=(4\pi (t+1))^{-\frac{n}{2}}e^{-\frac{|x-x_0|^2}{4(t+1)}}, \quad t\geq0, \;x\in\mathbb R^n.
\end{align}
From (\ref{special-solution}), we obtain by  direct computations that
\begin{align}\label{equ-ne-1}
\int_{\mathbb{R}^n}|v(1,x)|^2\,\mathrm dx=4^{-n}\pi^{-\frac{n}{2}}.
\end{align}
From (\ref{special-solution}), we also find that  for an arbitrarily fixed  $L>0$,
\begin{align}\label{equ-ne-2}
v(t,x)\leq(4\pi)^{-\frac{n}{2}}e^{-\frac{L^2}{16}}e^{-\frac{|x-x_0|^2}{16}}\;\; \text{for all}\;\;0\leq t\leq 1\;\,\text{and}\; x\in\mathbb R^n,\;\text{with}\;\;|x-x_0|\geq L.
\end{align}
By (\ref{equ-ne-2}), the above solution $v$ satisfies that
\begin{align}\label{equ-ne-3}
\int_0^1\int_{|x-x_0|\geq L}|v(t,x)|^2\,\mathrm dx\,\mathrm dt\leq (2\pi)^{-\frac{n}{2}}e^{-\frac{L^2}{8}}.
\end{align}
Meanwhile, by taking $T=1$ and $u=v$  in the observability inequality (\ref{ob-q}), we see that
\begin{align}\label{equ-ne-4}
\int_{\mathbb{R}^n}|v(1,x)|^2\,\mathrm dx&\leq C\int_0^1\int_{E}|v(t,x)|^2\,\mathrm dx\,\mathrm dt.
\end{align}
Here and in what follows, $C$ stands for the constant $C_{obs}(n,1,E)$ in (\ref{ob-q}).

Now, it follows from \eqref{equ-ne-1}, \eqref{equ-ne-4} and \eqref{equ-ne-3} that
\begin{align}\label{equ-ne-5}
4^{-n}\pi^{-\frac{n}{2}}\leq C\int_0^1\int_{E\bigcap B_L(x_0)}|v(t,x)|^2\,\mathrm dx\,\mathrm dt + C(2\pi)^{-\frac{n}{2}}e^{-\frac{L^2}{8}}.
\end{align}
Choose $L>0$ in such a way that
$$
C(2\pi)^{-\frac{n}{2}}e^{-\frac{L^2}{8}} \leq \frac{1}{2} 4^{-n}\pi^{-\frac{n}{2}}.
$$
Then by \eqref{equ-ne-5} and (\ref{special-solution}), we obtain that
\begin{align*}
\frac{1}{2}4^{-n}\pi^{-\frac{n}{2}}&\leq C\int_0^1\int_{E\bigcap B_L(x_0)}|v(t,x)|^2\,\mathrm dx\,\mathrm dt\nonumber\\
&= C\int_0^1\int_{E\bigcap B_L(x_0)}(4\pi (t+1))^{-n}e^{-\frac{|x-x_0|^2}{2(t+1)}}\,\mathrm dx\,\mathrm dt \nonumber\\
& \leq C\int_0^1\int_{E\bigcap B_L(x_0)}(4\pi)^{-n}\,\mathrm dx\,\mathrm dt  \leq C(4\pi)^{-n}
|E\bigcap B_L(x_0)|,
\end{align*}
from which, it follows that
\begin{eqnarray}\label{2.56Wang}
|E\bigcap B_L(x_0)|\geq (2C)^{-1}\pi^{\frac{n}{2}}.
\end{eqnarray}
Since $ B_L(x_0)\subset (x_0+2LQ)$, we see from (\ref{2.56Wang}) that
\begin{align*}
\left|E\bigcap (x_0+2LQ)\right|\geq(2C)^{-1}\pi^{\frac{n}{2}}.
\end{align*}
From this, as well as the choice of $L$,  we can find $L'>0$ and $\gamma>0$, which are independent of the choice of $x_0$, so that
\begin{align}\label{equ-ne-8}
\left|E\bigcap (x_0+L'Q)\right|\geq \gamma (L')^n.
\end{align}
Notice that $x_0$ in (\ref{equ-ne-8}) was arbitrarily taken from $\mathbb{R}^n$. Hence,
the set $E$ is $\gamma$-thick at scale $L'$. This ends the proof of
Lemma \ref{lem-thick-E}.
\end{proof}
\medskip
 We now on the position to prove Theorem \ref{equi-thm}.
 \begin{proof}[\textbf{Proof of Theorem \ref{equi-thm}}]
 We  can prove it in the following way:
$$
{\bf (i) \Rightarrow (ii) \Rightarrow (iii) \Rightarrow (iv) \Rightarrow (i).}
 $$
Indeed, the conclusions  {\bf (i) $\Rightarrow$ (ii)}, {\bf (ii) $\Rightarrow$ (iii)}, {\bf (iii) $\Rightarrow$ (iv)} and {\bf (iv) $\Rightarrow$ (i)} follow respectively from  Lemma \ref{lem-spec}, Lemma \ref{lem-hold}, Lemma \ref{lem-obs} and Lemma \ref{lem-thick-E}. This ends the proof of Theorem \ref{equi-thm}.
\end{proof}

\medskip

Tracking the constants in Lemma \ref{lem-spec}, Lemma \ref{lem-hold} and Lemma \ref{lem-obs}, we can easily get the following consequences of Theorem \ref{equi-thm}:
\begin{corollary}
Let $E\subset \mathbb{R}^n$ be a set of $\gamma$-thick at scale $L$ for some $\gamma>0$ and $L>0$.  Then  the following conclusions are true for a constant $C=C(n)>0$:
\begin{description}
                 \item[(a)] The set $E$ satisfies the H\"older-type interpolation \eqref{interpolation} with
$$
C_{Hold}(n,E,\theta)=\frac{C}{1-\theta}(1+L)^2\Big(1+\ln \frac{1}{\gamma}\Big)^2.
$$
                 \item[(b)] The set $E$ satisfies the observability inequality \eqref{ob-q} with
$$
C_{obs}(n,E,T)=e^{300(1+C)(1+L)^2\left(1+\ln\frac{1}{\gamma}\right)^2\left(1+\frac{1}{T}\right)}.
$$
\end{description}
\end{corollary}

\section{Weak interpolation and observability inequalities}\label{nons}
In this section, we  introduce several weak observability inequalities and interpolation
inequalities, where observations are made over a ball in $\mathbb{R}^n$.
One one hand, these inequalities
can be viewed as  extensions of (\ref{ob-q}) and (\ref{interpolation}) in some senses,
while on the other hand, they are independently interesting.

\subsection{Weak interpolation inequalities with observation on the unit ball}\label{07bao2}

 We begin with introducing two spaces. Given $a>0$ and $\nu>0$, we set
 $$
 L^2(e^{a|x|^\nu}\,\mathrm dx):=\{f: \mathbb{R}^n\rightarrow \mathbb{R}\;:\;
 f\;\mbox{is measurable and}\; \|f\|_{L^2(e^{a|x|^\nu}\,\mathrm dx)}<+\infty\},
 $$
 equipped with the norm:
  $$
\|f\|_{L^2(e^{a|x|^\nu}\,\mathrm dx)} :=  \left(\int_{\mathbb{R}^n}|f(x)|^2e^{a|x|^\nu}\,\mathrm dx\right)^{1/2},\;\;f\in L^2(e^{a|x|^\nu}\,\mathrm dx).
$$
 Given $\nu>0$, we set
 $$
 L^2(\langle x\rangle^\nu \,\mathrm dx):=
 \{f: \mathbb{R}^n\rightarrow \mathbb{R}\;:\;
 f\;\mbox{is measurable and}\; \|f\|_{L^2(\langle x\rangle^\nu \,\mathrm dx)}<+\infty\},
  $$
  equipped with the norm:
 $$
\|f\|_{L^2(\langle x\rangle^\nu \,\mathrm dx)} :=  \left(\int_{\mathbb{R}^n}|f(x)|^2\langle x\rangle^\nu \,\mathrm dx\right)^{1/2},\;\; f\in L^2(\langle x\rangle^\nu \,\mathrm dx).
$$
 Notice that any function in one of the above spaces decays along the radical direction.

  In this subsection,  we will build up some interpolation inequalities for
    solutions to (\ref{heat}),
 with initial data in $L^2(e^{a|x|^\nu}\,\mathrm dx)$ (or
 $L^2(\langle x\rangle^\nu \,\mathrm dx)$). In these inequalities, observations are made over the unit ball in $\mathbb{R}^n$ and at one time point.
     The purpose to study such observability has been explained in Subsection 1.3.
 Our main results about this subject are included in the following theorem:

\begin{theorem}\label{app-obs}
\begin{description}

  \item[(i)] There is $\theta=\theta(n)\in(0,1)$ and $C'=C'(n)$ such that
     for any $\varepsilon>0$, $T>0$ and $a>0$,
  $$
\int_{\mathbb{R}^n}|u(T,x)|^2\,\mathrm dx \leq C_1(a,T)\Big(\varepsilon\int_{\mathbb{R}^n}|u_0(x)|^2e^{a|x|}\,\mathrm dx  + \varepsilon^{-1}e^{\varepsilon^{-\frac{4|\ln \theta|}{a}}} \int_{B_1}|u(T,x)|^2\,\mathrm dx \Big),
$$
when $u$ solves \eqref{heat} with the initial condition
   $u(0,\cdot)=u_0(\cdot)\in L^2(e^{a|x|}\,\mathrm dx)$. Here,
   $$
C_1(a,T)=e^{C'\left(1+\frac{1}{T}+a +a^2T\right)}\sqrt{\Big(1+a^{-n}\Gamma\Big(\frac{a}{2|\ln \theta|}\Big)\Big)}.
$$

\item[(ii)]  There is $\theta=\theta(n)\in (0,1)$ and $C''=C''(n)$ so that
for any $\varepsilon\in(0,1)$,
$T>0$ and  $\nu\in(0,1]$,
  \begin{align*}
\int_{\mathbb{R}^n}|u(T,x)|^2\,\mathrm dx\leq  C_2(\nu,T)\Big(\varepsilon \int_{\mathbb{R}^n}|u_0|^2\langle x\rangle^\nu \,\mathrm dx + e^{e^{\left(3|\ln \theta|+1\right)\left(\frac{1}{\varepsilon}\right)^{\frac{1}{\nu}}}}
\int_{B_1}|u(T,x)|^2\,\mathrm dx\Big),
\end{align*}
when  $u$ solves \eqref{heat} with the initial condition
   $u(0,\cdot)=u_0(\cdot)\in L^2(\langle x\rangle^\nu \,\mathrm dx)$.
   Here,
 $$
 C_2(\nu,T)=(1+T^{\frac{\nu}{2}})e^{C''\left(1+\frac{1}{T}\right)}.
 $$

\end{description}
\end{theorem}

\medskip

\begin{remark}
$(a)$ The condition that $\nu\leq 1$ in (ii) of Theorem \ref{app-obs} is not necessary. We make this assumption only for the brevity of the statement of the theorem.
Indeed, from the definition of $L^2(\langle x\rangle^\nu \,\mathrm dx)$, we see that  $L^2(\langle x\rangle^\nu \,\mathrm dx) \hookrightarrow L^2(\langle x\rangle \,\mathrm dx)$ for any $\nu \geq 1$. From this and (ii) of Theorem \ref{app-obs}, one can easily check that
when $\nu>1$,
 any solution of \eqref{heat} satisfies that
\begin{align*}
\int_{\mathbb{R}^n}|u(T,x)|^2\,\mathrm dx\leq \left(1+T^{\frac{1}{2}}\right)e^{C''(n)\left(1+\frac{1}{T}\right)}\Big(\varepsilon \int_{\mathbb{R}^n}|u_0|^2\langle x\rangle^\nu \,\mathrm dx + e^{e^{\left(3|\ln \theta|+1\right)\frac{1}{\varepsilon}}}\int_{B_1}|u(T,x)|^2\,\mathrm dx\Big).
\end{align*}

$(b)$  \cite[Theorem 1]{EKPV-heat}
contains the following result: There is a universal constant $C>0$ so that for each $T>0$ and $R>0$,
\begin{eqnarray}\label{EKPV-uncertainty}
 \sup_{0\leq t\leq T} \Big\|e^{\frac{t|x|^2}{4\left(t^2+R^2\right)}} u(t) \Big\|_{L^2(\mathbb R^n)}
 \leq C \Big( \|u(0) \|_{L^2(\mathbb R^n)} +  \Big\|e^{\frac{T|x|^2}{4\left(T^2+R^2\right)}} u(T) \Big\|_{L^2(\mathbb R^n)} \Big),
\end{eqnarray}
when $u$ solves (\ref{heat}). The first inequality in Theorem \ref{app-obs} is comparable to
the above inequality (\ref{EKPV-uncertainty}). By our understanding, these two inequalities can be viewed as different versions of Hardy uncertainty principle. On one hand, the inequality (\ref{EKPV-uncertainty}) can be understood
as follows: From some information on a solution to (\ref{heat}) at infinity in $\mathbb{R}^n$  at two time points $0$ and $T$, one can know the behaviour of this solution at  infinity in $\mathbb{R}^n$ at each time  $t\in[0,T]$. On the other hand,  the first inequality in Theorem \ref{app-obs} can be explained in the following way: From some information on a solution to (\ref{heat}) at  infinity  in $\mathbb{R}^n$ at time $0$, and in the ball
$B_1$ in $\mathbb{R}^n$ at time $T$, one can know the behaviour of this solution
  at  infinity in $\mathbb{R}^n$ at time $T$.

  Similarly, we can compare   the second inequality in Theorem \ref{app-obs}
  with (\ref{EKPV-uncertainty}).
  It deserves to mention that we can only prove inequalities in  Theorem \ref{app-obs}
  for the pure heat equation (\ref{heat}), while \cite[Theorem 1]{EKPV-heat} gave the inequality
  (\ref{EKPV-uncertainty}) for heat equations with general potentials.

  $(c)$   The first inequality in Theorem \ref{app-obs} can also be understood as follows:
  If we know in advance that the initial datum of a solution to (\ref{heat}) is in the unit ball of $L^2(e^{a|x|}\,\mathrm dx)$, then  by observing this  solution in the unit ball of  $\mathbb{R}^n$ at time $T$,
  we can
  approximately recover this solution over $\mathbb{R}^n$ at the same time $T$, with the error $C_1(a,T)\varepsilon$.
  The second inequality in Theorem \ref{app-obs} can be explained in a very similar way.

\end{remark}

\medskip

To show  Theorem \ref{app-obs}, we need some preliminaries. We begin with
some auxiliary lemmas on the persistence of the heat semigroup in the  spaces
$L^2(e^{a|x|^\nu}\,\mathrm dx)$ and $L^2(\langle x\rangle^\nu \,\mathrm dx)$.
\begin{lemma}\label{lem-persis-1}
Let $a>0$ and $0<\nu\leq 1$. Then when $u_0\in L^2(e^{a|x|^\nu}\,\mathrm dx)$,
$$
\left\|e^{t\triangle}u_0\right\|_{L^2(e^{a|x|^\nu}\,\mathrm dx)}\leq 2^{\frac{n}{2}}e^{a^{\frac{2}{2-\nu}}t^{\frac{\nu}{2-\nu}}}
\|u_0\|_{L^2(e^{a|x|^\nu}\,\mathrm dx)}\quad\text{for all}\;\;t>0.
$$
\end{lemma}
\begin{proof}
Arbitrarily fix $a>0$, $0<\nu\leq 1$ and $u_0\in L^2(e^{a|x|^\nu}\,\mathrm dx)$.
Using the fundamental solution of \eqref{heat} and the definition of $L^2(e^{a|x|^\nu}\,\mathrm dx)$, we have that
\begin{eqnarray}\label{3.1Wang}
\left\|e^{t\triangle}u_0\right\|_{L^2(e^{a|x|^\nu}\,\mathrm dx)}
&=
\Big(\displaystyle\int_{\mathbb{R}^n}\Big(e^{\frac{a|x|^\nu}{2}}(4\pi t)^{-n/2}\displaystyle\int_{\mathbb{R}^n}e^{-\frac{|x-y|^2}{4t}}u_0(y)\,\mathrm dy\Big)^2dx\Big)^{\frac{1}{2}}.
\end{eqnarray}
Since
\begin{align*}
|x|^\nu \leq (|x-y|+|y|)^\nu\leq |x-y|^\nu+|y|^\nu\;\;\mbox{for all}\;\;x, y\in \mathbb{R}^n,
\end{align*}
(Here, we used  the elementary inequality: $(\tau+s)^\nu \leq \tau^\nu +s^\nu, \tau,s>0$.) it follows from (\ref{3.1Wang}) that
\begin{align}\label{equ-persis-3}
\left\|e^{t\triangle}u_0\right\|_{L^2(e^{a|x|^\nu}\,\mathrm dx)}
&\leq\Big (\displaystyle\int_{\mathbb{R}^n}\Big((4\pi t)^{-n/2}\int_{\mathbb{R}^n}e^{-\frac{|x-y|^2}{4t}
+\frac{a|x-y|^\nu}{2}}e^{\frac{a|y|^\nu}{2}}|u_0(y)|\,\mathrm dy\Big)^2dx\Big)^{\frac{1}{2}}\nonumber\\
&\leq \int_{\mathbb{R}^n}(4\pi t)^{-n/2}e^{-\frac{|x|^2}{4t}
+\frac{a|x|^\nu}{2}}dx \cdot
\Big(\int_{\mathbb{R}^n}\Big(e^{\frac{a|y|^\nu}{2}}u_0(y)
\Big)^2dy\Big)^{\frac{1}{2}}\nonumber\\
&= \int_{\mathbb{R}^n}(4\pi t)^{-n/2}e^{\frac{1}{4t}(-|x|^2+2ta|x|^\nu)}dx
\cdot \|u_0\|_{L^2(e^{a|x|^\nu}\,\mathrm dx)}.
\end{align}
Meanwhile, by the Young inequality:
$$
2ta|x|^\nu\leq \frac{|x|^2}{\frac{2}{\nu}}+\frac{(2ta)^{\frac{2}{2-\nu}}}{\frac{2}{2-\nu}}\leq \frac{1}{2}|x|^2+(2ta)^{\frac{2}{2-\nu}},
$$
we obtain that
\begin{align}\label{equ-persis-4}
\int_{\mathbb{R}^n}(4\pi t)^{-n/2}e^{\frac{1}{4t}(-|x|^2+2ta|x|^\nu)}dx &\leq \int_{\mathbb{R}^n}e^{\frac{(2ta)^{\frac{2}{2-\nu}}}{4t}}
(4\pi t)^{-n/2}e^{-\frac{|x|^2}{8t}}dx\nonumber\\
&= 2^{\frac{n}{2}}e^{\frac{(2ta)^{\frac{2}{2-\nu}}}{4t}}
\leq 2^{\frac{n}{2}}e^{a^{\frac{2}{2-\nu}}t^{\frac{\nu}{2-\nu}}}.
\end{align}
Now,  the desired inequality  follows from \eqref{equ-persis-3} and \eqref{equ-persis-4}.
This ends the proof of Lemma \ref{lem-persis-1}.
\end{proof}

\begin{remark}
The inequality in  Lemma \ref{lem-persis-1} does not hold for the case  when $\nu>1$.
Indeed, given $\nu>1$, let $u_0(x)=e^{-\frac{1}{2}|x|^\nu}\langle x\rangle^{-n}$, $x\in\mathbb R^n$.
It is clear that $u_0\in L^2(e^{|x|^\nu}\,\mathrm dx)$. However, we have that
for any $t>0$,  $e^{t\triangle}u_0\notin L^2(e^{|x|^\nu}\,\mathrm dx)$.
This can be proved as follows:  Arbitrarily fix $t>0$.
   By some direct calculations, we find that when $|x|\geq 2$,
\begin{align*}
\left(e^{t\triangle}u_0\right)(x)\geq C(4\pi t)^{-n/2}e^{-\frac{1}{4t}}e^{-\frac{1}{2}\left(|x|-\frac{1}{2}\right)^\nu}\langle |x|-{1}/{2}\rangle^{-n}\;\;\mbox{for some}\;\; C=C(n).
\end{align*}
This leads to that
\begin{align}\label{3.4wang}
\left\|e^{t\triangle}u_0\right\|_{L^2(e^{|x|^\nu}\,\mathrm dx)}\geq C(4\pi t)^{-n/2}e^{-\frac{1}{4t}} \Big(\int_{|x|\geq 2}e^{|x|^\nu-\left(|x|-\frac{1}{2}\right)^\nu}\langle |x|-{1}/{2}\rangle^{-2n}\,\mathrm dx\Big)^{1/2}.
\end{align}
Meanwhile, one can easily  find a constant $M>2$ so that
$$
|x|^\nu-\left(|x|-{1}/{2}\right)^\nu \geq \frac{\nu}{4}|x|^{\nu-1},\;\;\mbox{when}\;\;|x|\geq M.
$$
From this and  (\ref{3.4wang}), we obtain that
\begin{align*}
\left\|e^{t\triangle}u_0\right\|_{L^2(e^{|x|^\nu}\,\mathrm dx)}\geq C(4\pi t)^{-n/2}e^{-\frac{1}{4t}}\Big (\int_{|x|\geq M}e^{\frac{\nu}{4}|x|^{\nu-1}}\langle |x|-{1}/{2}\rangle^{-2n}\,\mathrm dx\Big)^{1/2}=\infty.
\end{align*}
\end{remark}

\medskip



\begin{lemma}\label{lem-persis-3}
Let $\nu\geq 0$. Then for any $u_0\in L^2(\langle x\rangle^\nu \,\mathrm dx)$,
$$
\|e^{t\triangle}u_0\|_{L^2(\langle x\rangle^\nu \,\mathrm dx)}\leq 4^{\nu+2}\Gamma({\nu}/{2}+n)\left(1+t^{\frac{\nu}{4}}\right)\|u_0\|_{L^2(\langle x\rangle^\nu \,\mathrm dx)}\;\;\mbox{for all}\;\; t>0.
$$
\end{lemma}
\begin{proof}
The proof is similar to that of Lemma \ref{lem-persis-1}. (See also \cite[Lemma B.6.1]{S}.)
\end{proof}

\begin{lemma}\label{lem-ua-1}
Given $s>0$, there is $C=C(n,s)$ so that when
 $f\in L^2(\mathbb{R}^n)$ satisfies that
 $\widehat{f}\in L^2(e^{a|\xi|^s}\,\mathrm d\xi)$
  for some  $a>0$,
$$
\left\|D^\alpha f\right\|_{L^\infty(\mathbb{R}^n)}\leq C^{|\alpha|+1}a^{-\frac{2|\alpha|+3n}{2s}}(\alpha!)^{\frac{1}{s}}\left\|
\widehat{f}(\xi)\right\|_{L^2(e^{a|\xi|^s}\,\mathrm d\xi)}\;\;\mbox{for each}\;\;\alpha\in \mathbb N^n.
$$
(Here,  we adopt the convention that $\alpha!=\alpha_1!\alpha_2!\cdots\alpha_n!$.)
\end{lemma}

\begin{remark}
From Lemma \ref{lem-ua-1}, we see that  if $f\in L^2(\mathbb{R}^n)$ satisfies that
$\widehat{f}\in L^2(e^{a|\xi|^s}\,\mathrm d\xi)$, with $s>0$ and $a>0$,
then $f$  is analytic, when $s=1$, while $f$
  is ultra-analytic, when $s>1$.
\end{remark}

\begin{proof}[\textbf{Proof of Lemma \ref{lem-ua-1}}]
Arbitrarily fix  $s>0$, $a>0$  and $f\in L^2(\mathbb{R}^n)$, with  $\widehat{f}\in L^2(e^{a|\xi|^s}\,\mathrm d\xi)$.
Then arbitrarily fix
$\alpha=(\alpha_1,\dots, \alpha_n)\in \mathbb N^n$, $\beta=(\beta_1,\dots,\beta_n)\in \mathbb N^n$ and $\gamma=(\gamma_1,\dots,\gamma_n)\in \mathbb N^n$, with $|\gamma|\leq n$.
Several facts are given in order.

\noindent {\it Fact One:}  By direct computations,  we see that
\begin{equation*}
\begin{split}
 \int_{\mathbb R^n} \left|\xi^{2\beta}\right| e^{-a|\xi|^s} \,\mathrm \,\mathrm d\xi
 &\leq \int_{\mathbb R^n} \left|\xi^{2\beta}\right| e^{-a\left(\Sigma_{i=1}^n|\xi_i|^s/n\right)} \,\mathrm \,\mathrm d\xi
 =\prod_{i=1}^n\int_{\mathbb R_{\xi_i}} \left|\xi_i\right|^{2\beta_i} e^{-a|\xi_i|^s/n} \,\mathrm \,\mathrm d\xi_i\\
 &= \prod_{i=1}^n 2 \int_0^\infty r^{2\beta_i} e^{-ar^s/n} \,\mathrm dr
 = \prod_{i=1}^n  2\left(\frac{n}{a}\right)^{\frac{2\beta_i+1}{s}} \int_0^\infty t^{\frac{2\beta_i+1}{s}-1} e^{-t} \,\mathrm \,\mathrm dt\\
& = 2^n\left(\frac{n}{a}\right)^{\frac{2|\beta|+n}{s}} \prod_{i=1}^n\Gamma\Big(\frac{2\beta_i+1}{s}\Big).
\end{split}
\end{equation*}
From this, we  obtain that
\begin{align}\label{equ-four-1.5}
\left(\int_{\mathbb R^n} |\xi^{2(\alpha+\gamma)}| e^{-a|\xi|^s} \,\mathrm \,\mathrm d\xi\right)^{1/2}\leq 2^{n/2}\left(\frac{n}{a}\right)^{\frac{2|\alpha|+3n}{2s}} \prod_{i=1}^n\sqrt{\Gamma\Big(\frac{2\alpha_i+2\gamma_i+1}{s}\Big)}.
\end{align}

\noindent {\it Fact Two:} By the Sobolev embedding $H^n(\mathbb{R}^n)\hookrightarrow L^\infty(\mathbb{R}^n)$, we can find  $C_1(n)>0$ so that
\begin{align}\label{equ-sup}
\left\|D^\alpha f\right\|_{L^\infty(\mathbb{R}^n)}\leq C_1(n)\sum_{\gamma\in \mathbb N^n,|\gamma|\leq n}\left\|D^{\alpha+\gamma} f\right\|_{L^2(\mathbb{R}^n)}.
\end{align}

\noindent {\it Fact Three:} By the Plancheral theorem and the H\"{o}lder inequality, we obtain that
\begin{align}\label{equ-Holder}
\left\|D^{\alpha+\gamma} f\right\|_{L^2(\mathbb{R}^n)}&= \left\|\xi^{\alpha+\gamma} \widehat{f}(\xi)\right\|_{L^2(\mathbb{R}^n)}\leq  \left\|\widehat{f}(\xi)e^{a|\xi|^s/2}\right\|_{L^2(\mathbb{R}^n)}\left(\int_{\mathbb R^n} |\xi^{2(\alpha+\gamma)}| e^{-a|\xi|^s} \,\mathrm \,\mathrm d\xi\right)^{1/2}.
\end{align}

\noindent {\it Fact Four:} There exists $C_2=C_2(n,s)$ so that
\begin{align}\label{equ-four-4}
\sqrt{\Gamma\Big(\frac{2\alpha_i+2\gamma_i+1}{s}\Big)}\leq C_2^{\alpha_i+1}{\Gamma^{1/s}(\alpha_i)}
=C_2^{\alpha_i+1}(\alpha_i!)^{\frac{1}{s}}.
\end{align}
The proof of (\ref{equ-four-4}) is as follows:
From  the Stirling formula, we have that
\begin{eqnarray}\label{3.9Wang}
\lim_{x\rightarrow +\infty} \frac{\Gamma(x)}{\sqrt{2\pi}e^{-x}x^{x+\frac{1}{2}}}=1.
\end{eqnarray}
From (\ref{3.9Wang}), we can find constants $M_1=M_1(s)$ and $C_3=C_3(n,s)$
so that
for all $\alpha_i>M_1$,
\begin{align}\label{equ-10-28-1}
\Gamma\Big(\frac{2\alpha_i+2\gamma_i+1}{s}\Big)
&\leq 2\sqrt{2\pi}e^{-\frac{2\alpha_i+2\gamma_i+1}{s}}
\Big(\frac{2\alpha_i+2\gamma_i+1}{s}\Big)
^{\frac{2\alpha_i+2\gamma_i+1}{s}+\frac{1}{2}}\nonumber\\
&= 2\sqrt{2\pi}\cdot e^{-\frac{2\alpha_i+2\gamma_i+1}{s}}
\Big(\frac{2\alpha_i+2\gamma_i+1}{s}\Big)^{\frac{2\gamma_i+1}{s}+\frac{1}{2}} \cdot  \Big(\frac{2\alpha_i+2\gamma_i+1}{s}\Big)^{\frac{2\alpha_i}{s}}\nonumber\\
&\leq 2\sqrt{2\pi}\sup_{x>0}e^{-x}x^{\frac{2\gamma_i+1}{s}+\frac{1}{2}} \cdot \Big(\frac{2\alpha_i+2\gamma_i+1}{s}\Big)^{\frac{2\alpha_i}{s}}\nonumber\\
&= 2\sqrt{2\pi}\cdot e^{-(\frac{2\gamma_i+1}{s}+\frac{1}{2})}
\Big(\frac{2\gamma_i+1}{s}+\frac{1}{2}\Big)^{\frac{2\gamma_i+1}{s}+\frac{1}{2}}  \cdot \Big(\frac{2\alpha_i+2n+1}{s}\Big)^{\frac{2\alpha_i}{s}}\nonumber\\
&\leq C_3^{\alpha_i}\alpha_i^{\frac{2\alpha_i}{s}}.
\end{align}
From (\ref{3.9Wang}), we can also find an absolute constant $M_2\geq 1$ so that for all $\alpha_i>M_2$,
\begin{align}\label{lowerbound-gamma}
\Gamma(\alpha_i)\geq 2^{-1}\sqrt{2\pi}e^{-\alpha_i}{\alpha_i}^{\alpha_i+\frac{1}{2}}.
\end{align}
According to \eqref{equ-10-28-1} and \eqref{lowerbound-gamma}, there is a constant $C_4(n,s)$ so that
\begin{align}\label{equ-four-2}
\frac{\sqrt{\Gamma\left(\frac{2\alpha_i+2\gamma_i+1}{s}\right)}}{\Gamma^{1/s}(\alpha_i)}\leq \left[C_4(n,s)\right]^{\alpha_i} \quad \textmd{for all }\alpha_i>M:=\max\{M_1, M_2\}.
\end{align}
Meanwhile, it is clear that there is a constant $C_5(n,s)$ so that
 \begin{align}\label{equ-four-3}
\frac{\sqrt{\Gamma\left(\frac{2\alpha_i+2\gamma_i+1}{s}\right)}}{\Gamma^{1/s}(\alpha_i)}\leq C_5(n,s)
\;\;\mbox{for all}\;\; \alpha_i\leq M.
\end{align}
(Here we agree that $\Gamma(0)=\infty$.)
Combining \eqref{equ-four-2} and \eqref{equ-four-3} leads to (\ref{equ-four-4}).

Inserting \eqref{equ-four-4} into \eqref{equ-four-1.5},
noticing that $|\gamma|\leq n$,
 we find that for some $C_6=C_6(n,s)$,
\begin{align}\label{equ-four-1.6}
\left(\int_{\mathbb R^n} \left|\xi^{2(\alpha+\gamma)}\right| e^{-a|\xi|^s} \,\mathrm \,\mathrm d\xi\right)^{1/2}&\leq 2^{n/2}\left(\frac{n}{a}\right)^{\frac{2|\alpha|+3n}{2s}} \prod_{i=1}^nC_2^{\alpha_i+1}\left(\alpha_i!\right)^{\frac{1}{s}}\nonumber\\
&\leq C_6^{|\alpha|+1}a^{-\frac{2|\alpha|+3n}{2s}}\left(\alpha!\right)^{\frac{1}{s}}.
\end{align}
Finally, it follows from \eqref{equ-sup}, \eqref{equ-Holder} and \eqref{equ-four-1.6} that
for some $C=C(n,s)$,
\begin{align*}
\left\|D^\alpha f\right\|_{L^\infty(\mathbb{R}^n)}&\leq C_1(n)\sum_{\gamma\in \mathbb N^n,|\gamma|\leq n}\left\|\widehat{f}(\xi)e^{a|\xi|^s/2}\right\|_{L^2(\mathbb{R}^n)} C_6^{|\alpha|+1}\left(\alpha!\right)^{\frac{1}{s}}\\
&\leq C^{|\alpha|+1}a^{-\frac{2|\alpha|+3n}{2s}}
\left(\alpha!\right)^{\frac{1}{s}}\left\|\widehat{f}(\xi)\right\|_{L^2(e^{a|\xi|^s}\,\mathrm d\xi)}.
\end{align*}
This ends the proof of Lemma \ref{lem-ua-1}.
\end{proof}

The next corollary is a consequence of  Lemma \ref{lem-ua-1}.
\begin{corollary}\label{lem-ua-2}
There is $C=C(n)>0$ so that when $f\in L^2(\mathbb{R}^n)$ satisfies that
  $\widehat{f}\in L^2(e^{a|\xi|^2}\,\mathrm d\xi)$ for some $a>0$,
$$
\left\|D^\alpha f\right\|_{L^\infty(\mathbb{R}^n)}\leq e^{C(1+b^2)\left(1+\frac{1}{a}\right)}
\frac{|\alpha|!}{b^{|\alpha|}}\left\|\widehat{f}(\xi)\right\|_{L^2(e^{a|\xi|^2}\,\mathrm d\xi)}\;\;\mbox{for all}\;\;b>0\;\;\mbox{and}\;\;\alpha\in \mathbb N^n.
$$
\end{corollary}

\medskip

\begin{remark}
Let $u_0\in L^2(\mathbb{R}^n)$  be arbitrarily given. Set
$u(t,x)=\left(e^{t\triangle}u_0\right)(x)$, $(t,x)\in (0,\infty)\times\mathbb{R}^n$.
Then $u$ is the solution of (\ref{heat}) with $u(0,\cdot)=u_0(\cdot)$.
Arbitrarily fix $t>0$.
 By applying Corollary \ref{lem-ua-2} (where $f(\cdot)=u(t,\cdot)$ and $a=2t$), we see that
   the radius of analyticity of $u(t,\cdot)$ (which is treated as a function of $x$) is independent of $t$.
It is an analogy result for solutions of the heat equation in a bounded domain with an analytic boundary (see \cite{AEWZ,EMZ}). This property
plays a very important role in the proof of the observability estimates
from measurable sets when using the telescope series method developed in \cite{AEWZ,EMZ}.
\end{remark}

\begin{proof}[\textbf{Proof of Corollary \ref{lem-ua-2}}]
Arbitrarily fix $a>0$ and $f\in L^2(\mathbb{R}^n)$ with $\widehat{f}\in L^2(e^{a|\xi|^2}\,\mathrm d\xi)$. Then arbitrarily fix $b>0$ and $\alpha\in \mathbb N^n$. According to
 Lemma \ref{lem-ua-1} (with $s=2$), there is $C'=C'(n)$  so that
\begin{align}\label{equ-ua-1}
\left\|D^\alpha f\right\|_{L^\infty(\mathbb{R}^n)}
&\leq {C'}^{|\alpha|+1}a^{-\frac{2|\alpha|+3n}{4}}
(\alpha!)^{\frac{1}{2}}\left\|\widehat{f}(\xi)\right\|_{L^2(e^{a|\xi|^2}\,\mathrm d\xi)}
\nonumber\\
&\leq g(|\alpha|)\frac{|\alpha|!}{b^{|\alpha|}}
\left\|\widehat{f}(\xi)\right\|_{L^2(e^{a|\xi|^2}\,\mathrm d\xi)},
\end{align}
where
$$
g(r)=C'a^{-3n/4}\left(bC'a^{-1/2}\right)^{r}\left(r!\right)^{-1/2},\;  r\in \mathbb{N}.
$$
To  estimate $g(r)$ pointwisely, we use \eqref{lowerbound-gamma} to
find that when $r>M_2\geq 1$ (where $M_2$ is given by \eqref{lowerbound-gamma}),
\begin{align}\label{equ-ua-2}
g(r)&\leq C'a^{-3n/4}\left(bC'a^{-1/2}\right)^{r}
\left(2^{-1}\sqrt{2\pi}e^{-r}r^{r+\frac{1}{2}}\right)^{-1/2}\nonumber\\
&\leq 2^{1/4}\pi^{-1/4}C'a^{-3n/4}\left(be^{1/2}C'a^{-1/2}\right)^rr^{-r/2}\nonumber\\
&\leq 2^{1/4}\pi^{-1/4}C'a^{-3n/4}
\sup_{r>0}~\left((be^{1/2}C'a^{-1/2})^rr^{-r/2}\right)\nonumber\\
&=2^{1/4}\pi^{-1/4}C'a^{-3n/4}e^{\frac{b^2C'^2}{2a}}.
\end{align}
Meanwhile, it is clear that when $r\leq M_2$,
\begin{align}\label{equ-ua-3}
g(r)\leq C'a^{-3n/4}\left(bC'a^{-1/2}+1\right)^{M_2}.
\end{align}
From \eqref{equ-ua-2} and \eqref{equ-ua-3}, it follows that
$$
g(r)\leq e^{C\left(1+b^2\right)\left(1+\frac{1}{a}\right)}\;\;\mbox{ for all}\;\;r=0,1,2,\dots,
$$
which,  together with \eqref{equ-ua-1}, yields the desired inequality.
This ends the proof of Corollary \ref{lem-ua-2}.
\end{proof}

\medskip

\medskip
To prove Theorem \ref{app-obs}, we also need the decomposition:
\begin{eqnarray}\label{3.17gwang}
\mathbb{R}^n = \bigcup_{j\geq 1}\Omega_j,\;\;\mbox{with}\;\;\Omega_j:=\{x\in \mathbb{R}^n: j-1\leq |x|< j\}.
\end{eqnarray}
The next lemma concerns with the propagation of  smallness for some real-analytic functions with respect to the above decomposition of $\mathbb R^n$.
\begin{lemma}\label{lem-smallpart}
 There are constants $C=C(n)>0$ and $\theta=\theta(n)\in (0,1)$ so that for any $a>0$, $j\geq1$ and $f\in L^2(\mathbb R^n)$ with $\widehat{f}\in L^2(e^{a|\xi|^2}\,\mathrm d\xi)$,
\begin{align*}
\int_{\Omega_{j+1}}|f|^2\,\mathrm dx\leq j^{(n-1)(1-\theta)} e^{C(1+\frac{1}{a})}\Big(\int_{\Omega_{j}}|f|^2\,\mathrm dx\Big)^\theta\Big(\int_{\mathbb{R}^n}|\widehat{f}|^2e^{a|\xi|^2}\,\mathrm d\xi\Big)^{1-\theta}.
\end{align*}
\end{lemma}
The proof of Lemma \ref{lem-smallpart} needs Corollary \ref{lem-ua-2} and the next lemma
which is quoted from \cite{AE} (see also {\cite[Theorem 4]{AEWZ}}), but is originally from \cite{S.V.}.

\begin{lemma}[{\cite[Theorem 1.3]{AE} }]\label{propagation}
Let $R>0$ and let  $f: B_{2R}\subset \mathbb{R}^n\rightarrow \mathbb{R}$ be real analytic in $B_{2R}$ verifying
$$
\left|D^\alpha f(x)\right|\leq M(\rho R)^{-|\alpha|}|\alpha|!,\;\;\mbox{when}\;\;  x\in B_{2R}\;\;
\mbox{and}\;\; \alpha\in\mathbb{N}^n
$$
for some positive numbers $M$ and $\rho\in (0,1]$.
Let $\omega\subset B_R$ be a subset of positive measure. Then there are two constants  $C=C(\rho,|\omega|/|B_R|)>0$ and $\theta=\theta(\rho,|\omega|/|B_R|)\in (0,1)$
so that
\begin{equation*}
\|f\|_{L^\infty(B_R)}\leq CM^{1-\theta}\Big(\frac{1}{|\omega|}\int_{\omega}|f(x)|\,dx\Big)^{\theta}.
\end{equation*}
\end{lemma}

We are now in the position to show Lemma \ref{lem-smallpart}.
\begin{proof}[\textbf{Proof of Lemma \ref{lem-smallpart}}]
Let $a>0$ and $j\geq 1$ be arbitrarily given.
Arbitrarily fix   $f\in L^2(\mathbb R^n)$ with $\widehat{f}\in L^2(e^{a|\xi|^2}\,\mathrm d\xi)$. The rest proof is divided into the following several steps.

\vskip 5pt
\textit{Step 1. The decompositions of $\Omega_j$ and $\Omega_{j+1}$ in the polar coordinates.}
In  the polar coordinate system, we have that
\begin{eqnarray}\label{3.18gwang}
\begin{cases}
\Omega_j=\big\{(r,\vartheta_1, \cdots, \vartheta_{n-1})\in [j-1,j) \times [0,2\pi)^{n-1}\big\},\\
\Omega_{j+1}=\big\{(r,\vartheta_1, \cdots, \vartheta_{n-1})\in [j,j+1) \times [0,2\pi)^{n-1}\big\}.
\end{cases}
\end{eqnarray}
  When $j$ is large, the distance between two points in $\Omega_j$ can be very large. This makes our studies on the propagation from $\Omega_j$ to $\Omega_{j+1}$ harder. To pass this barrier, we need to build up a suitable refinement for each $\Omega_j$.

 We set
\begin{eqnarray}\label{3.19gwang}
[0,2\pi] = \bigcup_{1\leq l\leq j} \Delta_l,\;\;\mbox{with}\;\;
\Delta_l := \left[\frac{l-1}{j}2\pi, \frac{l}{j}2\pi\right).
 \end{eqnarray}
  Given  $(k_1,\dots,k_{n-1})\in \mathbb{N}^{n-1}$ with $1\leq k_i\leq j$  ($i=1,\cdots, n-1$), we set
\begin{equation}\label{3.20gwang}
\begin{cases}
\Omega_{j;k_1,\cdots,k_{n-1}}=\left\{(r,\vartheta_1, \cdots, \vartheta_{n-1})\in \Omega_j~:~  \vartheta_1\in \Delta_{k_1}, \cdots, \vartheta_{n-1}\in \Delta_{k_{n-1}}\right\},\\
\Omega_{j+1;k_1,\cdots,k_{n-1}}=\left\{(r,\vartheta_1, \cdots, \vartheta_{n-1})\in \Omega_{j+1}~:~  \vartheta_1\in \Delta_{k_1}, \cdots, \vartheta_{n-1}\in \Delta_{k_{n-1}}\right\}.
\end{cases}
\end{equation}
Then one can easily check that for each $\hat j\in\{j,j+1\}$,
 \begin{eqnarray}\label{3.22GGWWSS}
 \Omega_{\hat j;k_1,\cdots,k_{n-1}}\bigcap \Omega_{\hat j;k_1',\cdots,k_{n-1}'}=\emptyset,\;\;\mbox{if}\;\; (k_1,\cdots,k_{n-1})\neq(k_1',\cdots,k_{n-1}'),
 \end{eqnarray}
 and
 \begin{eqnarray}\label{3.23GGWWSS}
 \Omega_{\hat j}=\bigcup \Omega_{\hat j;k_1,\cdots,k_{n-1}},
 \end{eqnarray}
 where the union is taken over all different $(k_1,\dots,k_{n-1})\in \mathbb{N}^{n-1}$, with $1\leq k_i\leq j$ for all $i=1,\dots,n-1$.

\vskip 5pt
In what follows, we write $d\left(\Omega_{j;k_1,\cdots,k_{n-1}}\right)$
for the diameter of $\Omega_{j;k_1,\cdots,k_{n-1}}$.

\textit{Step 2. To prove the following three properties:
\begin{itemize}
  \item[(O1)] There are constants  $c_1=c_1(n)$ and $c_2=c_2(n)$ so that
for  any $(k_1,\dots,k_{n-1})\in \mathbb{N}^{n-1}$, with $1\leq k_i\leq j$ ($i=1,\dots,n-1$),
\begin{align}\label{equ-ua-4}
c_1V_n \leq \left|\Omega_{j;k_1,\cdots,k_{n-1}}\right|\leq c_2V_n.
\end{align}
  \item[(O2)] We have that for   any $(k_1,\dots,k_{n-1})\in \mathbb{N}^{n-1}$ with $1\leq k_i\leq j$ ($i=1,\dots,n-1$),
\begin{align}\label{equ-ua-5}
d\left(\Omega_{j;k_1,\cdots,k_{n-1}}\right)&=d\left(\Omega_{j;1,\cdots,1}\right)
:=\sup_{x,x'\in \Omega_{j;1,\cdots,1}}|x-x'|.
\end{align}
  \item[(O3)] We have that for any fix  $(k_1,\dots,k_{n-1})\in \mathbb{N}^{n-1}$ with $1\leq k_i\leq j$ ($i=1,\dots, n-1$),
\begin{align}\label{equ-ua-6}
d\left(\Omega_{j;k_1,\cdots,k_{n-1}}\right)\leq 2\pi \sqrt{\sum_{1\leq i\leq n}i^2}\leq 2\pi n^{\frac{3}{2}};
\end{align}
\begin{align}\label{equ-ua-7}
d\left(\Omega_{j+1;k_1,\cdots,k_{n-1}}\right)\leq 2\pi \frac{j+1}{j} \sqrt{\sum_{1\leq i\leq n}i^2}\leq 4\pi n^{\frac{3}{2}}.
\end{align}
\end{itemize}}
\noindent To see (O1), we use the definitions of $\Omega_{j;k_1,\cdots,k_{n-1}}$
 and $\Omega_j$  to find that
$$
\left|\Omega_{j;k_1,\cdots,k_{n-1}}\right|
=\frac{1}{j^{n-1}}\left|\Omega_j\right|=V_n\frac{j^n-(j-1)^n}{j^{n-1}},
$$
which leads to (\ref{equ-ua-4}).

 The conclusion (O2) follows immediately from the definitions
of $\Omega_{j;k_1,\cdots,k_{n-1}}$ and $d\left(\Omega_{j;k_1,\cdots,k_{n-1}}\right)$.

 To show (\ref{equ-ua-6}) in (O3), we let  $(r, \vartheta_1, \cdots, \vartheta_{n-1})$ and $(r, \vartheta_1', \cdots, \vartheta_{n-1}')$ be the polar coordinates of $x=(x_1,\dots,x_n)$ and  $x'=(x_1',\dots,x_n')$, respectively. Then we have that
 \begin{eqnarray}\label{3.21Wang}
 x, x'\in \Omega_{j;1,\cdots,1}\;\;\Longleftrightarrow\;\;j-1\leq r< j,\;\; 0\leq \vartheta_l,\; \vartheta_l'< \frac{2\pi}{j}\;\;\mbox{for all}\;\;l=1,\cdots, n-1.
 \end{eqnarray}
 Notice that the connection between $(x_1,\dots,x_n)$ and $(r,\vartheta_1,\dots,\vartheta_{n-1})$ is as:
\begin{equation*}\left\{
\begin{split}
&x_1=r\cos \vartheta_1,\\
&x_2=r\sin \vartheta_1\cos \vartheta_2,\\
&\cdots\\
&x_{n-1}=r\sin\vartheta_1\sin \vartheta_2 \cdots \sin \vartheta_{n-2}\cos \vartheta_{n-1},\\
&x_n=r\sin\vartheta_1\sin \vartheta_2 \cdots \sin \vartheta_{n-2}\sin \vartheta_{n-1}.
\end{split}\right.
\end{equation*}
Then, by the mean value theorem, we have that for some $\zeta\in (0,2\pi/j)$,
$$
|x_1-x_1'|=r|\cos \vartheta_1 - \cos \vartheta_1'|=r|\sin \zeta \cdot(\vartheta_1-\vartheta_1')|\leq j |\sin \zeta| \frac{2\pi}{j}\leq 2\pi.
$$
By inserting suitable terms and using the mean value theorem, we have that
$$
|x_2-x_2'|\leq r\Big(|\sin\vartheta_1\cos\vartheta_2-\sin\vartheta_1\cos\vartheta_2'|
+|\sin\vartheta_1\cos\vartheta_2'-\sin\vartheta_1'\cos\vartheta_2'|\Big)\leq
 2\pi\cdot 2.
$$
Similarly, we can verify that
$$
|x_i-x_i'|\leq 2\pi\cdot i\;\;\mbox{ for all}\;\;i=3,\dots,n.
 $$
 These, along with (\ref{equ-ua-5}),
lead to (\ref{equ-ua-6}).

The inequality (\ref{equ-ua-7}) in (O3) can be proved in the same way. The reason that the factor $\frac{j+1}{j}$ appears in (\ref{equ-ua-7}) is as follows: Since $1\leq k_i\leq j$ ($i=1,\dots, n-1$), we see from the definition of $\Omega_{j+1;1,\cdots,1}$ that
\begin{eqnarray*}
 x, x'\in \Omega_{j+1;1,\cdots,1}\;\;\Longleftrightarrow\;\;j\leq r< j+1,\;\; 0\leq \vartheta_l,\; \vartheta_l'< \frac{2\pi}{j}\;\;\mbox{for all}\;\;l=1,\cdots, n-1.
 \end{eqnarray*}
 (The above is comparable with (\ref{3.21Wang}).)

\vskip 5pt
\textit{Step 3. To prove that there are constants $C=C(n)>0$ and $\theta=\theta(n)\in (0,1)$ (both independent of $a$, $j$ and $f$) so that
\begin{align*}
\int_{\Omega_{j+1;k_1,\cdots,k_{n-1}}}|f|^2\,\mathrm dx\leq e^{C(1+\frac{1}{a})}\Big(\int_{\Omega_{j;k_1,\cdots,k_{n-1}}}|f|^2\,\mathrm dx\Big)^\theta\Big(\int_{\mathbb{R}^n}|\widehat{f}|^2e^{a|\xi|^2}\,\mathrm d\xi\Big)^{1-\theta}.
\end{align*}}
Since $\Omega_{j;k_1,\cdots,k_{n-1}} \bigcup \Omega_{j+1;k_1,\cdots,k_{n-1}}$ is connected
(see (\ref{3.20gwang})),
it follows from \eqref{equ-ua-6} and \eqref{equ-ua-7}
that
\begin{align*}
d\left(\Omega_{j;k_1,\cdots,k_{n-1}} \bigcup \Omega_{j+1;k_1,\cdots,k_{n-1}}\right)\leq d\left(\Omega_{j;k_1,\cdots,k_{n-1}}\right)
+d\left(\Omega_{j+1;k_1,\cdots,k_{n-1}}\right)\leq 6\pi n^{\frac{3}{2}}.
\end{align*}
Thus,  there exists  $\widetilde{x}\in \mathbb{R}^n$ such that
\begin{align}\label{equ-ua-8}
\Omega_{j;k_1,\cdots,k_{n-1}} \bigcup \Omega_{j+1;k_1,\cdots,k_{n-1}} \subset {B_{R_0}(\widetilde{x})},\;\;\mbox{with}\;\;R_0=7\pi n^{\frac{3}{2}}.
\end{align}
According to  Corollary \ref{lem-ua-2} where $b=R_0$,  there is $\widehat{C}=\widehat{C}(n)>0$  such that for all  $ \alpha\in \mathbb N^n$,
\begin{align}\label{equ-ua-9}
\left\|D^\alpha f\right\|_{L^\infty(\mathbb{R}^n)}
\leq e^{\widehat{C}\left(1+\frac{1}{a}\right)}
\frac{|\alpha|!}{R_0^{|\alpha|}}\left\|\widehat{f}(\xi)\right\|_{L^2(e^{a|\xi|^2}\,\mathrm d\xi)}.
\end{align}
By \eqref{equ-ua-9}, as well as (\ref{equ-ua-8}),  we can apply Lemma \ref{propagation},
where
$$
\rho=1,\;R=R_0,\;B_R=B_{R_0}(\widetilde{x}),\;B_{2R}=B_{2R_0}(\widetilde{x}),\;
\omega=\Omega_{j;k_1,\cdots,k_{n-1}},
\;M=e^{\widehat{C}\left(1+\frac{1}{a}\right)}
\left\|\widehat{f}(\xi)\right\|_{L^2(e^{a|\xi|^2}\,\mathrm d\xi)},
$$
 to find constants
$C_0=C_0(n)>0$ and $\theta=\theta(n)\in (0,1)$ so that
\begin{align}\label{equ-ua-10}
\|f\|_{L^2(B_{R_0}(\widetilde{x}))}&\leq C_0\|f\|^\theta_{L^2\left(\Omega_{j;k_1,\cdots,k_{n-1}}\right)} \left(e^{\widehat{C}\left(1+\frac{1}{a}\right)}
\left\|\widehat{f}(\xi)\right\|_{L^2(e^{a|\xi|^2}\,\mathrm d\xi)}\right)^{1-\theta}\nonumber\\
&\leq C_0e^{\widehat{C}\left(1+\frac{1}{a}\right)}
\|f\|^\theta_{L^2\left(\Omega_{j;k_1,\cdots,k_{n-1}}\right)} \left\|\widehat{f}(\xi)\right\|_{L^2\left(e^{a|\xi|^2}\,\mathrm d\xi\right)}^{1-\theta}.
\end{align}
(Here, we used (\ref{equ-ua-4}) and a coordinate translation.)

Finally, the desired inequality of this step follows from \eqref{equ-ua-10}
and  \eqref{equ-ua-8}.

\vskip 5pt
\textit{Step 4. To prove the inequality of this lemma}

\noindent From (\ref{3.22GGWWSS}) and (\ref{3.23GGWWSS}), we see that
$\Omega_j$ is the disjoint union of all $\Omega_{j;k_1,\cdots,k_{n-1}}$ with different
$(k_1,\dots,k_{n-1})\in\mathbb{N}^{n-1}$
satisfying  $1\leq k_i\leq j$ for all $i=1,\dots,n-1$.
Meanwhile, by (\ref{3.18gwang}),  (\ref{3.19gwang}) and (\ref{3.20gwang}),
one can also check that
$\Omega_{j+1}$ is the disjoint union of all $\Omega_{j+1;k_1,\cdots,k_{n-1}}$ with different
$(k_1,\dots,k_{n-1})\in\mathbb{N}^{n-1}$
satisfying  $1\leq k_i\leq j$ for all $i=1,\dots,n-1$.
These, along with  Lemma~\ref{lem-smallpart}, yield that  for some $C=C(n)>0$ and $\theta=\theta(n)$,
\begin{align}\label{equ-ua-11}
&\int_{\Omega_{j+1}}|f|^2\,\mathrm dx = \sum\int_{\Omega_{j+1;k_1,\cdots,k_{n-1}}}|f|^2\,\mathrm dx\nonumber\\
&\leq \sum e^{C\left(1+\frac{1}{a}\right)}\Big(\int_{\Omega_{j;k_1,\cdots,k_{n-1}}}|f|^2\,\mathrm dx\Big)^\theta\Big(\int_{\mathbb{R}^n}|\widehat{f}|^2e^{a|\xi|^2}\,\mathrm d\xi\Big)^{1-\theta}\nonumber\\
&\leq e^{C\left(1+\frac{1}{a}\right)}\Big(\sum\int_{\Omega_{j;k_1,\cdots,k_{n-1}}}|f|^2\,\mathrm dx\Big)^\theta\Big(\sum\int_{\mathbb{R}^n}|\widehat{f}|^2e^{a|\xi|^2}\,\mathrm d\xi\Big)^{1-\theta}\nonumber\\
&=j^{(n-1)(1-\theta)}e^{C\left(1+\frac{1}{a}\right)}\Big(\int_{\Omega_{j}}|f|^2\,\mathrm dx\Big)^\theta\Big(\int_{\mathbb{R}^n}|\widehat{f}|^2e^{a|\xi|^2}\,\mathrm d\xi\Big)^{1-\theta},
\end{align}
where the sums are taken over all different $(k_1,\dots,k_{n-1})\in \mathbb{N}^{n-1}$
with $1\leq k_i\leq j$ for all $i=1,\dots,n-1$. (Notice that there are $j^{n-1}$ such $(k_1,\dots,k_{n-1})$). Hence, the desired conclusion follows from (\ref{equ-ua-11}). This ends the proof of Lemma \ref{lem-smallpart}.
\end{proof}

Based on Lemma \ref{lem-smallpart}, we can have the next propagation result which will be used later.

\begin{lemma}\label{lem-ring}
 There exist constants $C=C(n)>0$ and $\theta=\theta(n)\in (0,1)$ so that
 for any  $a>0$ and $j\geq1$,
 \begin{align*}
\int_{\Omega_{j+1}}|f|^2\,\mathrm dx\leq j^{n-1}e^{C\left(1+\frac{1}{a}\right)}\Big(\int_{B_{1}}|f|^2\,\mathrm dx\Big)^{\theta^j}\Big(\int_{\mathbb{R}^n}|\widehat{f}|^2e^{a|\xi|^2}\,\mathrm d\xi\Big)^{1-\theta^j},
\end{align*}
when  $f\in L^2(\mathbb R^n)$ satisfies that $\widehat{f}\in L^2(e^{a|\xi|^2}\,\mathrm d\xi)$.
\end{lemma}
\begin{proof}
Arbitrarily fix $a>0$ and $j\geq 1$. And then arbitrarily fix $f\in L^2(\mathbb R^n)$ with $\widehat{f}\in L^2(e^{a|\xi|^2}\,\mathrm d\xi)$.
From Lemma \ref{lem-smallpart}, we can use the induction method to verify  that
\begin{align}\label{equ-ua-12}
\int_{\Omega_{j+1}}|f|^2\,\mathrm dx&\leq \Big(j(j-1)^{\theta}(j-2)^{\theta^2}\cdots 2^{\theta^{j-2}}1^{\theta^{j-1}}\Big)^{(n-1)(1-\theta)}\nonumber\\
&\;\;\;\;\times e^{C(1+\frac{1}{a})(1+\theta+\cdots +\theta^{j-1})}\Big(\int_{\Omega_{1}}|f|^2\,\mathrm dx\Big)^{\theta^j}\Big(\int_{\mathbb{R}^n}|\widehat{f}|^2e^{a|\xi|^2}\,\mathrm d\xi\Big)^{1-\theta^j}\nonumber\\
&\leq j^{n-1}e^{\frac{C}{1-\theta}(1+\frac{1}{a})}\Big(\int_{\Omega_{1}}|f|^2\,\mathrm dx\Big)^{\theta^j}\Big(\int_{\mathbb{R}^n}|\widehat{f}|^2e^{a|\xi|^2}\,\mathrm d\xi\Big)^{1-\theta^j}.
\end{align}
Since $\Omega_1=B_1$ and $\theta=\theta(n)\in (0,1)$, the desired conclusion in the lemma follows from \eqref{equ-ua-12}. This ends the proof of Lemma \ref{lem-ring}.
\end{proof}

The next proposition plays a very important role in the proof of Theorem \ref{app-obs}.
\begin{proposition}\label{thm-ua}
There exist constants $C=C(n)>0$ and $\theta=\theta(n)\in (0,1)$ so that for any $a>0$, $t>0$
and  $\varepsilon>0$,
\begin{align*}
\int_{\mathbb{R}^n}e^{-a|x|}|f|^2\,\mathrm dx\leq e^{C\left(1+\frac{1}{t}+a\right)}\Big(1+a^{-n}\Gamma\Big(\frac{a}{2|\ln \theta|}\Big)\Big)\Big(\varepsilon \int_{\mathbb{R}^n}|\widehat{f}|^2e^{t|\xi|^2}\,\mathrm d\xi + e^{\varepsilon^{-\frac{2|\ln \theta|}{a}}} \int_{B_1}|f|^2\,\mathrm dx \Big),
\end{align*}
when  $f\in L^2(\mathbb R^n)$ satisfies $\widehat{f}\in L^2(e^{t|\xi|^2}\,\mathrm d\xi)$.
\end{proposition}
To prove Proposition \ref{thm-ua}, we need the following  result quoted from \cite{WWZ}:
\begin{lemma}[{\cite[Lemma 3.1]{WWZ}}]\label{lem-series-sum}
Let $a>0$, $b\in(0,1)$ and $\theta\in(0,1)$. Then
\begin{eqnarray*}
 \sum_{k=1}^\infty   b^{\theta^k} e^{-ak}
 \leq  \frac{e^a}{|\ln\theta|}  \Gamma\Big(\frac{a}{|\ln b|}\Big) |\ln x|^{-\frac{a}{|\ln\theta|}}.
\end{eqnarray*}
\end{lemma}

\begin{proof}[\textbf{Proof of Proposition \ref{thm-ua}}]
Arbitrarily fix $a>0$ and $t>0$. And then arbitrarily fix $f\in L^2(\mathbb R^n)$ satisfies $\widehat{f}\in L^2(e^{t|\xi|^2}\,\mathrm d\xi)$.
It suffices to show the inequality in Proposition \ref{thm-ua} for
the above fixed $a$, $t$, $f$ and any $\varepsilon>0$.
Without loss of generality, we can assume that
\begin{eqnarray}\label{3.31GGwang}
A:=\int_{\mathbb{R}^n}|\widehat{f}|^2e^{t|\xi|^2}\,\mathrm d\xi\neq 0\;\;\mbox{and}\;\; B:=\int_{B_{1}}|f|^2\,\mathrm dx\neq 0.
\end{eqnarray}
For otherwise, when $A=0$, we have that $f=0$ over $\mathbb{R}^n$, thus the desired inequality  is trivial; while when $B=0$, we can use
the analyticity of $f$ (which follows from Corollary \ref{lem-ua-2}) to see that
 $f=0$ over $\mathbb{R}^n$ and then the desired inequality   is trivial again.

By (\ref{3.17gwang}), we have that
\begin{align}\label{equ-thm-ua-2}
\int_{\mathbb{R}^n}e^{-a|x|}|f|^2\,\mathrm dx&= \int_{B_{1}}e^{-a|x|}|f|^2\,\mathrm dx+ \sum_{j\geq 1} \int_{\Omega_{j+1}}e^{-a|x|}|f|^2\,\mathrm dx\nonumber\\
&\leq \int_{B_{1}}|f|^2\,\mathrm dx+ \sum_{j\geq 1} \int_{\Omega_{j+1}}e^{-aj}|f|^2\,\mathrm dx.
\end{align}
We now  estimate the last term of \eqref{equ-thm-ua-2}. According to Lemma \ref{lem-ring}, there is $C_1=C_1(n)>0$ and $\theta=\theta(n)\in (0,1)$ so that
\begin{align}\label{equ-thm-ua-3}
\sum_{j\geq 1} \int_{\Omega_{j+1}}e^{-aj}|f|^2\,\mathrm dx&\leq e^{C_1\left(1+\frac{1}{t}\right)}\sum_{j\geq 1}j^{n-1}e^{-aj}\left(\int_{B_{1}}|f|^2\,\mathrm dx\right)^{\theta^j}\left(\int_{\mathbb{R}^n}|\widehat{f}|^2e^{t|\xi|^2}\,\mathrm d\xi\right)^{1-\theta^j}\nonumber\\
&\leq e^{C_1\left(1+\frac{1}{t}\right)}n!\left({2}/{a}\right)^n\sum_{j\geq 1}e^{-\frac{a}{2}j}A\left({B}/{A}\right)^{\theta^j},
\end{align}
where  $A$ and $B$ are given by (\ref{3.31GGwang}).
In the proof of (\ref{equ-thm-ua-3}), we used   the inequality:
$$
j^ne^{-\frac{a}{2}j}\leq n!\left({2}/{a}\right)^n\;\;\mbox{for all}\;\;j\geq 1.
$$
Meanwhile, by Lemma \ref{lem-series-sum} (with $b=B/A$), we have that
\begin{align}\label{equ-thm-ua-4}
\sum_{j\geq 1}e^{-\frac{a}{2}j}A({B}/{A})^{\theta^j}\leq \frac{e^{\frac{a}{2}}}{|\ln \theta|}\Gamma\Big(\frac{a}{2|\ln \theta|}\Big)A|\ln ({B}/{A})|^{-\frac{a}{2|\ln \theta|}}.
\end{align}

About $A/B$, there are only two possibilities: either $A/B>e$ or $A/B\leq e$.

In the first case when  $A/B> e$, we claim that
\begin{align}\label{equ-thm-ua-5}
A|\ln ({B}/{A})|^{-\frac{a}{2|\ln \theta|}} \leq \varepsilon A + e^{\varepsilon^{-\frac{2|\ln \theta|}{a}}} B\;\;\mbox{for all}\;\;\varepsilon>0.
\end{align}
In fact, when $\varepsilon$ satisfies that
$$
A|\ln ({B}/{A})|^{-\frac{a}{2|\ln \theta|}} \leq \varepsilon A,
 $$
 \eqref{equ-thm-ua-5} is trivial. One the other hand,  when $\varepsilon>0$ satisfies that
 $$
 A|\ln ({B}/{A})|^{-\frac{a}{2|\ln \theta|}} > \varepsilon A,
 $$
  we have that
$$
{A}/{B}<e^{\varepsilon^{-\frac{2|\ln \theta|}{a}}}.
$$
 This, along with the fact that
${A}/{B}> e$, yields that
$$
A|\ln ({B}/{A})|^{-\frac{a}{2|\ln \theta|}}\leq A \leq B\cdot {A}/{B}\leq e^{\varepsilon^{-\frac{2|\ln \theta|}{a}}} B\;\;\mbox{for all}\;\;\varepsilon>0.
$$
Thus, we have proved \eqref{equ-thm-ua-5}.

Inserting \eqref{equ-thm-ua-5} into \eqref{equ-thm-ua-4} leads to
\begin{align}\label{equ-thm-ua-6}
\sum_{j\geq 1}e^{-\frac{a}{2}j}A({B}/{A})^{\theta^j}\leq \frac{e^{\frac{a}{2}}}{|\ln \theta|}\Gamma\Big(\frac{a}{2|\ln \theta|}\Big)\Big(\varepsilon A + e^{\varepsilon^{-\frac{2|\ln \theta|}{a}}} B \Big)\;\;\mbox{for all}\;\;\varepsilon>0.
\end{align}
Now it follows from \eqref{equ-thm-ua-2},\eqref{equ-thm-ua-3}
 and \eqref{equ-thm-ua-6} that for some $C_2=C_2(n)>0$,
\begin{align*}
\int_{\mathbb{R}^n}e^{-a|x|}|f|^2\,\mathrm dx &\leq e^{C_1(1+\frac{1}{t})}n!\Big(1+({2}/{a})^n\frac{e^{\frac{a}{2}}}{|\ln \theta|}\Gamma\Big(\frac{a}{2|\ln \theta|}\Big)\Big)\Big(\varepsilon A + e^{\varepsilon^{-\frac{2|\ln \theta|}{a}}} B \Big)\\
&\leq e^{C_2(1+\frac{1}{t}+a)}\Big(1+a^{-n}\Gamma\Big(\frac{a}{2|\ln \theta|}\Big)\Big)\Big(\varepsilon A
+ e^{\varepsilon^{-\frac{2|\ln \theta|}{a}}} B \Big)\;\;\mbox{for any}\;\;\varepsilon>0.
\end{align*}
This proves the desired inequality  for the first case that $A/B>e$.

In the second case  where  $A/B\leq e$, we derive directly that
\begin{align*}
\int_{\mathbb{R}^n}e^{-a|x|}|f|^2\,\mathrm dx&  \leq \int_{\mathbb{R}^n}|f|^2\,\mathrm dx \leq \int_{\mathbb{R}^n}|\widehat{f}|^2e^{t|\xi|^2}\,\mathrm d\xi \leq e\int_{B_1}|f|^2\,\mathrm dx\\
& \leq e\left(\varepsilon \int_{\mathbb{R}^n}|\widehat{f}|^2e^{t|\xi|^2}\,\mathrm d\xi + e^{\varepsilon^{-\frac{2|\ln \theta|}{a}}} \int_{B_1}|f|^2\,\mathrm dx \right)\;\;\mbox{for any}\;\;\varepsilon>0.
\end{align*}
This proves the desired inequality  for the second case that $A/B\leq e$.

Hence, we end the proof of Proposition \ref{thm-ua}.
\end{proof}


We now are on the position to show  Theorem \ref{app-obs}.

\begin{proof}[\textbf{Proof of Theorem \ref{app-obs}}] (i). Arbitrarily fix  $u_0\in L^2(e^{a|x|}\mathrm dx)$. Let $u(T,x)=(e^{T\triangle}u_0)(x)$, $x\in \mathbb{R}^n$. By the H\"{o}lder inequality, we have that
\begin{align}\label{equ-thm2-1}
\int_{\mathbb{R}^n}|u(T,x)|^2\,\mathrm dx \leq \Big(\int_{\mathbb{R}^n}|u(T,x)|^2e^{a|x|}\,\mathrm dx\Big)^{1/2}\Big(\int_{\mathbb{R}^n}|u(T,x)|^2e^{-a|x|}\,\mathrm dx\Big)^{1/2}.
\end{align}
We will estimate the two terms on right side of \eqref{equ-thm2-1} one by one. For the first term, we apply Lemma \ref{lem-persis-1} (with $\nu=1$) to obtain that
\begin{align}\label{equ-thm2-2}
\int_{\mathbb{R}^n}|u(T,x)|^2e^{a|x|}\,\mathrm dx \leq 2^{n}e^{2a^2T}\int_{\mathbb{R}^n}|u_0(x)|^2e^{a|x|}\,\mathrm dx.
\end{align}
To estimate  the second term (on right side of \eqref{equ-thm2-1}), we first notice that
$\widehat{e^{T\Delta}u_0}\in L^2(e^{T|\xi|^2}d\xi)$, since $u_0\in L^2(e^{a|x|}\mathrm dx)\subset L^2(\mathbb{R}^n)$. Thus,
we can apply Proposition \ref{thm-ua} (with $f=e^{T\triangle}u_0$ and $t=2T$) to find
$C=C(n)>0$ and $\theta=\theta(n)\in (0,1)$ so that for each $\varepsilon>0$,
\begin{align}\label{equ-thm2-3}
\int_{\mathbb{R}^n}e^{-a|x|}|u(T,x)|^2\,\mathrm dx \leq C(T,a,n) \Big(\varepsilon \int_{\mathbb{R}^n}|u_0(x)|^2\,\mathrm d x + e^{\varepsilon^{-\frac{2|\ln \theta|}{a}}} \int_{B_1}|u(T,x)|^2\,\mathrm dx \Big),
\end{align}
 where
 $$
 C(T,a,n) = e^{C(1+\frac{1}{T}+a)}\Big(1+a^{-n}\Gamma\Big(\frac{a}{2|\ln \theta|}\Big)\Big).
 $$

Inserting \eqref{equ-thm2-2} and \eqref{equ-thm2-3} into \eqref{equ-thm2-1}, we get that for each $\varepsilon>0$
\begin{align}\label{equ-thm2-4}
\int_{\mathbb{R}^n}&|u(T,x)|^2\,\mathrm dx
 \leq \widehat{C} \Big(\int_{\mathbb{R}^n}|u_0(x)|^2e^{a|x|}\,\mathrm dx\Big)^{1/2}\Big(\varepsilon \int_{\mathbb{R}^n}|u_0(x)|^2\,\mathrm d x + e^{\varepsilon^{-\frac{2|\ln \theta|}{a}}} \int_{B_1}|u(T,x)|^2\,\mathrm dx \Big)^{1/2}\nonumber\\
& \leq \widehat{C} \Big(\varepsilon^{1/2}\int_{\mathbb{R}^n}|u_0(x)|^2e^{a|x|}\,\mathrm dx\Big)^{1/2}\Big(\varepsilon^{1/2} \int_{\mathbb{R}^n}|u_0(x)|^2\,\mathrm d x + \varepsilon^{-1/2}e^{\varepsilon^{-\frac{2|\ln \theta|}{a}}} \int_{B_1}|u(T,x)|^2\,\mathrm dx \Big)^{1/2}\nonumber\\
&\leq 2^{-1} \widehat{C}\Big(\varepsilon^{1/2}\int_{\mathbb{R}^n}|u_0(x)|^2e^{a|x|}\,\mathrm dx + \varepsilon^{1/2} \int_{\mathbb{R}^n}|u_0(x)|^2\,\mathrm dx+\varepsilon^{-1/2}e^{\varepsilon^{-\frac{2|\ln \theta|}{a}}} \int_{B_1}|u(T,x)|^2\,\mathrm dx\Big)\nonumber\\
&\leq \widehat{C}\Big(\varepsilon^{1/2}\int_{\mathbb{R}^n}|u_0(x)|^2e^{a|x|}\,\mathrm dx  + \varepsilon^{-1/2}e^{\varepsilon^{-\frac{2|\ln \theta|}{a}}} \int_{B_1}|u(T,x)|^2\,\mathrm dx \Big),
\end{align}
where
$$
\widehat{C}=\widehat{C}(T,a,n)=2^{n/2}e^{a^2T}\sqrt{C(T,a,n)}.
$$
Since $\varepsilon>0$ can be arbitrary taken, we replace  $\varepsilon$ by $\varepsilon^2$ in \eqref{equ-thm2-4} to get the desired conclusion in (i) of Theorem \ref{app-obs}.

\medskip

(ii). Arbitrarily fix $u_0\in L^2(\langle x\rangle^\nu \mathrm dx)$. Let $u(T,x)=\left(e^{T\triangle}u_0\right)(x)$, $x\in \mathbb{R}^n$.
Three facts are given in order.
Fact One. Using the inequality:
$$
1\leq \varepsilon |x|^\nu + e^{({1}/{\varepsilon})^{\frac{1}{\nu}}}e^{-|x|}
 \;\;\mbox{for all}\;\; \varepsilon>0\;\;\mbox{and}\;\; x\in \mathbb{R}^n,
$$
we find that
\begin{align}\label{equ-thm2-5}
\int_{\mathbb{R}^n}|u(T,x)|^2\,\mathrm dx\leq \varepsilon \int_{\mathbb{R}^n}|u(T,x)|^2|x|^\nu \,\mathrm dx + e^{({1}/{\varepsilon})^{\frac{1}{\nu}}}\int_{\mathbb{R}^n}|u(T,x)|^2e^{-|x|}\,\mathrm dx.
\end{align}
Fact Two. Since $0<\nu\leq 1$, we can apply  Lemma \ref{lem-persis-3} to find
 $C_1=C_1(n)$ so that
\begin{align}\label{equ-thm2-6}
\int_{\mathbb{R}^n}|u(T,x)|^2|x|^\nu \,\mathrm dx &\leq \Big(4^{\nu+2}\Gamma({\nu}/{2}+n)\left(1+T^{\frac{\nu}{4}}\right)\|u_0\|_{L^2(\langle x\rangle^\nu \,\mathrm dx)}\Big)^2\nonumber\\
&\leq C_1 \left(1+T^{\frac{\nu}{2}}\right)\int_{\mathbb{R}^n}|u_0(x)|^2\langle x\rangle^\nu \,\mathrm dx.
\end{align}
Fact Three. We can use \eqref{equ-thm2-3} (with  $a=1$ and $\varepsilon=\mu$) to find $C_2=C_2(n)$ and $\theta=\theta(n)\in (0,1)$ so that for all $\mu>0$,
 \begin{align}\label{equ-thm2-7}
\int_{\mathbb{R}^n}e^{-|x|}|u(T,x)|^2\,\mathrm dx\leq e^{C_2\left(1+\frac{1}{T}\right)}\Big(\mu \int_{\mathbb{R}^n}|u_0(x)|^2\,\mathrm dx + e^{\mu^{-2|\ln \theta|}} \int_{B_1}\left|u(T,x)\right|^2\,\mathrm dx \Big).
 \end{align}

To continue the proof, we arbitrarily fix $\varepsilon\in(0,1)$. We will first use Fact Three, and then use Fact One and Fact Two. By taking $\mu = \varepsilon e^{-(\frac{1}{\varepsilon})^{\frac{1}{\nu}}}$ in \eqref{equ-thm2-7}, we obtain that
\begin{align}\label{equ-thm2-8}
&e^{({1}/{\varepsilon})^{\frac{1}{\nu}}}\int_{\mathbb{R}^n}e^{-|x|}|u(T,x)|^2\,\mathrm dx\nonumber\\
&\leq e^{({1}/{\varepsilon})^{\frac{1}{\nu}}}e^{C_2\left(1+\frac{1}{T}\right)}\Big(\mu \int_{\mathbb{R}^n}|u_0(x)|^2\,\mathrm dx + e^{\mu^{-2|\ln \theta|}}\int_{B_1}|u(T,x)|^2\,\mathrm dx\Big)\nonumber\\
&= e^{C_2\left(1+\frac{1}{T}\right)}\Big(\varepsilon \int_{\mathbb{R}^n}|u_0(x)|^2\,\mathrm dx + b_\varepsilon\int_{B_1}|u(T,x)|^2\,\mathrm dx\Big),
\end{align}
where
$$
b_\varepsilon = \exp \Big[{({1}/{\varepsilon})^{\frac{1}{\nu}} + (\varepsilon^{-2}e^{2({1}/{\varepsilon})^{\frac{1}{\nu}}})^{|\ln\theta|}}\Big].
$$
Meanwhile, one can directly check
 the following two inequalities:
\begin{align}\label{equ-thm2-9}
s^2 \leq e^s\;\;\mbox{for all}\;\;s>0;
\end{align}
\begin{align}\label{equ-thm2-10}
s+e^{3|\ln \theta|s}\leq e^{(3|\ln \theta|+1)s}\;\;\mbox{for all}\;\;s>0.
\end{align}
Choosing  $s=\varepsilon^{-1}$ and
$s=({1}/{\varepsilon})^{\frac{1}{\nu}}$
in \eqref{equ-thm2-9} and  \eqref{equ-thm2-10} respectively, using  $0<\nu\leq 1$, we find
that
\begin{align}\label{equ-10-28-2}
b_\varepsilon
&\leq \exp\Big[{({1}/{\varepsilon})^{\frac{1}{\nu}} + (e^{{1}/{\varepsilon}}
e^{2({1}/{\varepsilon})^{\frac{1}{\nu}}})^{|\ln\theta|}}\Big]\nonumber\\
&\leq \exp\Big[{({1}/{\varepsilon})^{\frac{1}{\nu}} + (e^{3({1}/{\varepsilon})^{\frac{1}{\nu}}})^{|\ln\theta|}}\Big]
\nonumber\\
&\leq \exp\Big[{e^{(3|\ln \theta|+1)({1}/{\varepsilon})^{\frac{1}{\nu}}}}\Big].
\end{align}
Combining \eqref{equ-thm2-8} and \eqref{equ-10-28-2} leads to  that
\begin{align}\label{equ-10-28-3}
&e^{({1}/{\varepsilon})^{\frac{1}{\nu}}}\int_{\mathbb{R}^n}e^{-|x|}|u(T,x)|^2\,\mathrm dx\nonumber\\
&\leq  e^{C_2(n)\left(1+\frac{1}{T}\right)}\Big(\varepsilon \int_{\mathbb{R}^n}|u_0(x)|^2\,\mathrm dx + e^{e^{(3|\ln \theta|+1)({1}/{\varepsilon})^{\frac{1}{\nu}}}}\int_{B_1}|u(T,x)|^2\,\mathrm dx\Big).
\end{align}

Finally, inserting \eqref{equ-thm2-6} and \eqref{equ-10-28-3} into \eqref{equ-thm2-5}, we obtain that for some  $C_3=C_3(n)$,
\begin{eqnarray*}
&&\int_{\mathbb{R}^n}|u(T,x)|^2\,\mathrm dx\nonumber\\
&\leq& \Big(C_1 (1+T^{\frac{\nu}{2}})+e^{C_2(1+\frac{1}{T})}\Big)\Big(\varepsilon \int_{\mathbb{R}^n}|u_0(x)|^2\langle x\rangle^\nu \,\mathrm dx + e^{e^{(3|\ln \theta|+1)({1}/{\varepsilon})^{\frac{1}{\nu}}}}\int_{B_1}|u(T,x)|^2\,\mathrm dx\Big)\\
&\leq& (1+T^{\frac{\nu}{2}})e^{C_3(1+\frac{1}{T})}\Big(\varepsilon \int_{\mathbb{R}^n}|u_0(x)|^2\langle x\rangle^\nu \,\mathrm dx + e^{e^{(3|\ln \theta|+1)({1}/{\varepsilon})^{\frac{1}{\nu}}}}\int_{B_1}|u(T,x)|^2\,\mathrm dx\Big),
\end{eqnarray*}
which leads to the desired conclusion in (ii) of Theorem \ref{app-obs}.

Hence, we end the proof of  Theorem \ref{app-obs}.
\end{proof}

\subsection{Weak observability inequalities with observations on  balls}\label{07bao1}

According to Theorem \ref{equi-thm}, it is impossible to recover a solution of \eqref{heat} by observing it over a ball. Thus, two  interesting questions arise.  First, can  we  recover a solution of \eqref{heat} over a ball by observing it on another ball?
Second, can we have observability inequalities with observations over balls  for solutions of \eqref{heat} with some kind of
initial values?
The answer to the first question is almost negative, while we give partially positive  answer for the second question.
The first  main result of this subsection is
 stated as follows:

\begin{theorem}\label{prop-1}
\begin{description}
  \item [(i)] There is an absolute positive constant $C$ so that for all $T>0$ and  $0<r'<r$,
\begin{align*}
\int_{B_{r'}}u^2(T,x)\,\mathrm dx\leq\left(\frac{1}{T}+\frac{Cn}{(r-r')^2}\right)\int_0^T\int_{B_r}u^2(t,x)\,\mathrm dx\,\mathrm dt,  \;\;\mbox{when}\;\; u\;\;\mbox{solves}\;\; \eqref{heat}.
\end{align*}

  \item [(ii)] Given $T>0$ and  $r'>r>0$, there is no constant $C=C(T,r',r,n)$  so that
  \begin{align*}
\int_{B_{r'}}u^2(T,x)\,\mathrm dx\leq C\int_0^T\int_{B_r}u^2(t,x)\,\mathrm dx\,\mathrm dt
\;\;\mbox{for  any solution}\;\; u\;\;\mbox{to}\;\;\eqref{heat}.
\end{align*}

\end{description}
\end{theorem}

\begin{proof}

(i) Arbitrarily fix $T>0$ and $0<r'<r$. Arbitrarily fix a solution $u$ to \eqref{heat}. Let $u(0,x)=u_0(x)$, $x\in \mathbb{R}^n$.
Choose  a $C^2$ function  $\varphi$
 on $\mathbb{R}^n$ so that for some absolute constant $C>0$,
\begin{eqnarray}\label{3.50wGGng}
0\leq \varphi(x)\leq 1\;\;\mbox{over}\;\;\mathbb{R}^n;\;
\left|D^\alpha \varphi(x)\right|\leq C \left(r-r'\right)^{-|\alpha|}\;\;\mbox{for all}\;\;\alpha\in \mathbb{N}^n,
\;\;\mbox{with}\;\; |\alpha|\leq 2,
\end{eqnarray}
and so that
$$
\varphi(x)=
\left\{
  \begin{array}{ll}
    1, & \hbox{$x\in B_{r'},$} \\
    0, & \hbox{$x\in B_r^c.$}
  \end{array}
\right.
$$
 Set $v=\varphi u$. Then $v$ satisfies
\begin{align}\label{equ-v}
v_t-\triangle v = -2\nabla \varphi \cdot \nabla u - \triangle \varphi u\quad\text{in}\;\;\mathbb R^+\times\mathbb R^n,\;\;\;\; v(0,\cdot)=\varphi(\cdot) u_0(\cdot)\;\;\text{in}\;\;\mathbb R^n.
\end{align}
Multiplying  \eqref{equ-v} by $tv$ leads to
\begin{align}\label{equ-v-1}
\frac{1}{2}\left(tv^2\right)_t-\frac{1}{2}v^2-tv\triangle v = -2tu\nabla \varphi \cdot \nabla v + t\left(2|\nabla \varphi|^2-\varphi\triangle \varphi\right) u^2.
\end{align}
Integrating \eqref{equ-v-1} over $(0,T)\times \mathbb{R}^n$, we have that
\begin{eqnarray}\label{equ-1}
&&\frac{1}{2}\int_{\mathbb{R}^n}Tv^2(T,x)\,\mathrm dx-\frac{1}{2}\int_0^T\int_{\mathbb{R}^n}v^2(t,x)\,\mathrm dx\,\mathrm dt+\int_0^T\int_{\mathbb{R}^n}t\left|\nabla v(t,x)\right|^2\,\mathrm dx\,\mathrm dt\\
&=&\int_0^T\int_{\mathbb{R}^n}-2tu\nabla \varphi(x) \cdot \nabla v(t,x)\,\mathrm dx\,\mathrm dt\nonumber\\
&&+\int_0^T\int_{\mathbb{R}^n}t\left(2\left|\nabla \varphi(x)\right|^2-\varphi(x)\triangle \varphi(x)\right) u^2(t,x)\,\mathrm dx\,\mathrm dt.\nonumber
\end{eqnarray}
Meanwhile, by the H\"{o}lder inequality, we find that
\begin{align}\label{equ-2}
&\Big|\int_0^T\int_{\mathbb{R}^n}-2tu(t,x)\nabla \varphi(x) \cdot \nabla v(t,x)\,\mathrm dx\,\mathrm dt\Big|\nonumber\\
&\leq \int_0^T\int_{\mathbb{R}^n}t\left|\nabla v(t,x)\right|^2\,\mathrm dx\,\mathrm dt + \int_0^T\int_{\mathbb{R}^n}t\left|\nabla \varphi(x)\right|^2u^2(t,x)\,\mathrm dx\,\mathrm dt.
\end{align}
Inserting \eqref{equ-2} into \eqref{equ-1} leads to that
\begin{align}\label{equ-3}
&T\int_{\mathbb{R}^n}v^2(T,x)\,\mathrm dx\nonumber\\
&\leq \int_0^T\int_{\mathbb{R}^n}v^2(t,x)\,\mathrm dx\,\mathrm dt+2\int_0^T\int_{\mathbb{R}^n}t\left(3|\nabla \varphi(x)|^2-\varphi(x)\triangle \varphi(x)\right) u^2(t,x)\,\mathrm dx\,\mathrm dt.
\end{align}
Since $v=u$ on $B_{r'}$ and $v=0$ on $B^c_r$, we can use  \eqref{equ-3} and
(\ref{3.50wGGng}) to     get that
\begin{align*}
&\int_{B_{r'}}u^2(T,x)\,\mathrm dx\\
&\leq \frac{1}{T}\int_0^T\int_{B_r}u^2(t,x)\,\mathrm dx\,\mathrm dt+\frac{2}{T}\int_0^T\int_{B_r\backslash B_{r'}}t\left(3|\nabla \varphi(x)|^2-\varphi(x)\triangle \varphi(x)\right) u^2(t,x)\,\mathrm dx\,\mathrm dt\nonumber\\
&\leq \left(\frac{1}{T}+\frac{8Cn}{(r-r')^2}\right)\int_0^T\int_{B_r}u^2(t,x)\,\mathrm dx\,\mathrm dt,
\end{align*}
which leads to the desired conclusion in (i) of  Theorem \ref{prop-1}.

\medskip

(ii)   By contradiction, we suppose that there were  $T>0$,  $r'>r>0$ and $C=C(T,r',r,n)>0$  so that
  \begin{align}\label{equ-contra-1}
\int_{B_{r'}}u^2(T,x)\,\mathrm dx\leq C\int_0^T\int_{B_r}u^2(t,x)\,\mathrm dx\,\mathrm dt
\;\;\mbox{for any solution}\;\;u\;\;\mbox{to}\;\; \eqref{heat}.
\end{align}
We  would use a constructive method to derive a contradiction with  (\ref{equ-contra-1}).
For this purpose, we define, for  each $k\geq1$,
$$
u_k(t,x)=\frac{1}{(4\pi (t+1))^{n/2}}e^{-\frac{|x_1-k|^2+|x'|^2}{4(t+1)}},\quad (t,x)=(t,x_1,x')\in[0,\infty)\times\mathbb R\times\mathbb R^{n-1}.
$$
One can easily check that  $u_k$ is the solution of \eqref{heat} with initial value:
$$
u_k(0,x)=\frac{1}{(4\pi)^{n/2}}e^{-\frac{|x_1-k|^2+|x'|^2}{4}},\quad x=(x_1,x')\in\mathbb R\times\mathbb R^{n-1}.
$$
It is clear that $\{u_k(0,\cdot)\}_{k\geq 1}$  is uniformly bounded in $L^2(\mathbb{R}^n)$.

We next show that when $k$ is large enough,  $u_k$ does not satisfy  \eqref{equ-contra-1}
(which leads to a contradiction).  To this end, we need  two estimates:
\begin{align}\label{equ-contra-2}
\int_0^T\int_{B_r}u_k^2(t,x)\,\mathrm dx\,\mathrm dt\leq \frac{r^nTV_n}{(4\pi (T+1))^n}e^{-\frac{(k-r)^2}{2(T+1)}},\; \mbox{when}\; k>r+\sqrt{2n(T+1)};
\end{align}
\begin{align}\label{equ-contra-3}
\int_{B_{r'}}u_k^2(T,x)\,\mathrm dx\geq \frac{1}{(4\pi (T+1))^n}e^{-\frac{(k-r-\frac{\sigma}{3})^2}{2(T+1)}}({\sigma}/{3})^nV_n,\; \mbox{with}\;
\sigma=r'-r,\;\;\mbox{when}\;\;k>r'.
\end{align}

To show \eqref{equ-contra-2}, we obtain from a direct computation that
$$
\partial_t(\ln u_k(t,x)) = -\frac{n}{2}\frac{1}{t+1}+\frac{|x_1-k|^2+|x'|^2}{4(t+1)^2}
\; \mbox{for all}\;\; t\geq 0,\; x=(x_1,x')\in \mathbb{R}^n,
$$
from which, it follows that when $ t\geq 0$ and $x=(x_1,x')\in \mathbb{R}^n$,
$$
\partial_t(\ln u_k(t,x))>0 \Longleftrightarrow |x_1-k|^2+|x'|^2>2n(t+1).
$$
This implies that
$$
k>r+\sqrt{2n(T+1)}  \Longrightarrow \partial_t(\ln u_k(t,x))>0\;\;\mbox{for all}\;\;(t,x)\times(0,T)\times B_r.
$$
From the above, we see that when $k>r+\sqrt{2n(T+1)}$, we have that for each
$x\in B_{r}$, $u_k(t,x)$ is an increasing function of $t$ on $[0,T]$.
Hence, when $k>r+\sqrt{2n(T+1)}$,
\begin{align*}
\int_0^T\int_{B_r}u_k^2(t,x)\,\mathrm dx\,\mathrm dt &\leq \int_0^T\int_{B_r}u_k^2(T,x)\,\mathrm dx\,\mathrm dt=\int_0^T\int_{B_r}\frac{1}{(4\pi (T+1))^n}e^{-\frac{|x_1-k|^2+|x'|^2}{2(T+1)}}\,\mathrm dx\,\mathrm dt\\
&=\frac{T}{(4\pi (T+1))^n}\int_{B_r}e^{-\frac{|x_1-k|^2+|x'|^2}{2(T+1)}}\,\mathrm dx\leq \frac{T}{(4\pi (T+1))^n}\int_{B_r}e^{-\frac{|k-r|^2}{2(T+1)}}\,\mathrm dx\\
&\leq \frac{r^nTV_n}{(4\pi (T+1))^n}e^{-\frac{(k-r)^2}{2(T+1)}},
\end{align*}
which leads to  \eqref{equ-contra-2}.

We next  show \eqref{equ-contra-3}. Given $b>0$,  use $B_r(b,0')$ to denote the ball centered at $(b,0,0,\cdots,0)$ and of radius $r$, namely
$$
B_r(b,0'):=\left\{x\in \mathbb{R}^n: |x_1-b|^2+|x'|^2\leq r^2\right\}.
$$
Let  $k>r'$ and $\sigma=r'-r$. Then
\begin{align}\label{equ-7}
\int_{B_{r'}}u_k^2(T,x)\,\mathrm dx&\geq\int_{B_{r'}\bigcap B_{k-r-\frac{\sigma}{3}}(k,0')}u_k^2(T,x)\,\mathrm dx\nonumber\\
&=\int_{B_{r'}\bigcap B_{k-r-\frac{\sigma}{3}}(k,0')}\frac{1}{(4\pi (T+1))^n}e^{-\frac{|x_1-k|^2+|x'|^2}{2(T+1)}}\,\mathrm dx\nonumber\\
&\geq \frac{1}{(4\pi (T+1))^n}e^{-\frac{\left(k-r-\frac{\sigma}{3}\right)^2}{2(T+1)}}\left|B_{r'}\bigcap B_{k-r-\frac{\sigma}{3}}(k,0')\right|.
\end{align}
Meanwhile, it is clear that
$$
B_{\frac{\sigma}{3}}(r+{2\sigma}/{3},0')\subset B_{r'}\bigcap B_{k-r-\frac{\sigma}{3}}(k,0'),
$$
 which leads to that
\begin{align}\label{equ-contra-4}
\left|B_{r'}\bigcap B_{k-r-\frac{\sigma}{3}}(k,0')\right|\geq \left|B_{\frac{\sigma}{3}}(r+{2\sigma}/{3},0')\right|=({\sigma}/{3})^nV_n.
\end{align}
Now \eqref{equ-contra-3} follows from \eqref{equ-7} and \eqref{equ-contra-4} at once.

 From  \eqref{equ-contra-2} and \eqref{equ-contra-3}, we find that when $k>\max\{r+\sqrt{2n(T+1)}, r'\}$,
\begin{align}\label{3.61GGwang}
\frac{\displaystyle\int_0^T\int_{B_r}u_k^2(t,x)\,\mathrm dx\,\mathrm dt}{\displaystyle\int_{B_{r'}}u_k^2(T,x)\,\mathrm dx}
&\leq \displaystyle\frac{\frac{r^nTV_n}
{(4\pi (T+1))^n}e^{-\frac{(k-r)^2}{2(T+1)}}}{\frac{1}{(4\pi (T+1))^n}e^{-\frac{\left(k-r-\frac{\sigma}{3}\right)^2}
{2(T+1)}}\left(\frac{\sigma}{3}\right)^nV_n}\nonumber\\
&=T\Big(\frac{3r}{\sigma}\Big)^ne^{\frac{1}{2(T+1)}
\left(\left(k-r-\frac{\sigma}{3}\right)^2-(k-r)^2\right)}\nonumber\\
&=T\Big(\frac{3r}{\sigma}\Big)^ne^{\frac{1}{2(T+1)}
\left((\frac{\sigma}{3})^2-\frac{2\sigma}{3}(k-r)\right)}\nonumber\\
&=T\Big(\frac{3r}{r'-r}\Big)^ne^{\frac{1}{2(T+1)}
\big(\big(\frac{r'-r}{3}\big)^2-\frac{2(r'-r)}{3}(k-r)\big)}.
\end{align}
Since
 $$
 (\frac{r'-r}{3})^2-\frac{2(r'-r)}{3}(k-r)\rightarrow-\infty,\;\;\mbox{as}\;\;k\rightarrow\infty,
 $$
  we see from (\ref{3.61GGwang}) that
$$
\lim_{k\rightarrow + \infty}\frac{\displaystyle\int_0^T\int_{B_r}u_k^2(t,x)\,\mathrm dx\,\mathrm dt}{\displaystyle\int_{B_{r'}}u_k^2(T,x)\,\mathrm dx}=0,
$$
from which, it follows that   $u_k$
 does not satisfy \eqref{equ-contra-1}, when  $k$ is large enough. This shows
 the conclusion  in (ii) of  Theorem \ref{prop-1}.

 Hence, we end the proof of Theorem \ref{prop-1}.
 \end{proof}
\medskip

The next corollary is a direct consequence of Theorem \ref{prop-1}.
\begin{corollary}\label{corollary3.2}
Given $\nu>0$, $T>0$ and $r>0$, there is no constant $C=C(T,r,\nu,n)>0$ so that
  \begin{align}\label{ob-wg-3}
\int_{\mathbb{R}^n}u^2(T,x)\rho(x)\,\mathrm dx\leq C\int_0^T\int_{B_r}u^2(t,x)\,\mathrm dx\,\mathrm dt\;\;\mbox{for any}\;\;u\;\;\mbox{solves}\;\;\eqref{heat},
\end{align}
where either $\rho(x)=\langle x\rangle^{-\nu}$, $x\in\mathbb R^n$, or $\rho(x)=e^{-|x|}$, $x\in\mathbb R^n$.
\end{corollary}


\begin{remark}
It was announced in \cite[p. 384]{CMV} (without proof) that given a bounded interval
$E$, there is no positive weight function $\rho$ such that
$$
\int_0^\infty |u(T,x)|^2\rho(x)\,\mathrm dx \leq C \int_0^T\int_E |u(t,x)|^2\,\mathrm dx\,\mathrm dt
$$
for all solutions of the heat equation in the physical space $(0, \infty)$.  The above
Corollary \ref{prop-1} presents a similar result for the heat equation in the physical space $\mathbb{R}^n$.
\end{remark}


The second  main result of this subsection is
 stated as follows:
\begin{theorem}\label{obs-special-data}
\begin{description}
  \item[(i)] There is a generic constant $C$ so that  for any $T>0$, $M>r>0$
  and $u_0\in L^2(\mathbb{R}^n)$ with $supp\, u_0 \subset B_r$,
\begin{align*}
\int_{\mathbb{R}^n}|u(T,x)|^2\,\mathrm dx\leq \left(\frac{1}{T}+\frac{Cn}{(M-r)^2}\right)\int_0^T\int_{B_M}|u(t,x)|^2\,\mathrm dx\,\mathrm dt,
\end{align*}
where $u$ is the solution to \eqref{heat} with $u(0,\cdot)=u_0(\cdot)$.
  \item[(ii)] Assume that $0\leq u_0\in L^1(\mathbb{R}^n)$ so that
$$
\int_{B_r}u_0(x)\,\mathrm dx\geq \mu \int_{\mathbb{R}^n}u_0(x)\,\mathrm dx\;\;\mbox{for some}\;\; r>0\;\;\mbox{and}\;\; \mu\in (0,1).
$$
Then for any $T>0$, $M>0$ and any solution $u$ to \eqref{heat} with $u(0,\cdot)=u_0(\cdot)$,
$$
\int_{\mathbb{R}^n} |u(T,x)|^2\,\mathrm dx \leq \frac{2^{\frac{n}{2}+1}\pi^{\frac{n}{2}}T^{\frac{n}{2}-1}}{V_n (r\wedge M)^n \mu^2}e^{\frac{4r^2}{T}}\int_{0}^T \int_{B_M} |u(t,x)|^2\,\mathrm dx\,\mathrm dt.
$$
Here, $r\wedge M := \min\{r, M\}$.
\end{description}
\end{theorem}

\begin{proof}
(i) The proof is similar to that of (i) of Theorem \ref{prop-1}. Arbitrarily fix  $T>0$, $M>r>0$ and $u_0\in L^2(\mathbb{R}^n)$ with $supp\, u_0 \subset B_r$. Write $u$ for the solution to \eqref{heat} with $u(0,\cdot)=u_0(\cdot)$.
    Choose  a $C^2$ function $\varphi$ over $\mathbb{R}^n$  so that
    for some absolute constant $C>0$,
\begin{eqnarray}\label{3.63GGwangg}
0\leq \varphi(x)\leq 1\;\;\mbox{over}\;\;\mathbb{R}^n;\;
|D^\alpha \varphi(x)|\leq C (M-r)^{-|\alpha|}\;\;\mbox{for all}\;\;\alpha\in \mathbb{N}^n,
\;\;\mbox{with}\;\; |\alpha|\leq 2
\end{eqnarray}
   and  so that
$$
\varphi(x)=
\left\{
  \begin{array}{ll}
    0, & \hbox{$x\in B_{r},$} \\
    1, & \hbox{$x\in B_M^c.$}
  \end{array}
\right.
$$
Set $v=\varphi u$. Multiplying \eqref{equ-v} by $v$, we find that
\begin{align}\label{equ-special-0}
\frac{1}{2}\left(v^2\right)_t-v\triangle v = -2u\nabla \varphi \cdot \nabla v + \left(2|\nabla \varphi|^2-\varphi\triangle \varphi\right) u^2.
\end{align}
Integrating \eqref{equ-special-0} over $(0,T)\times\mathbb{R}^n$, we obtain that
\begin{eqnarray}\label{equ-special-1}
&&\frac{1}{2}\int_{\mathbb{R}^n}v^2(T,x)\,\mathrm dx-\frac{1}{2}\int_{\mathbb{R}^n}v^2(0,x)\,\mathrm dx+\int_0^T\int_{\mathbb{R}^n}|\nabla v(t,x)|^2\,\mathrm dx\,\mathrm dt\\
&=&\int_0^T\int_{\mathbb{R}^n}-2u(t,x)\nabla \varphi(x) \cdot \nabla v(t,x)\,\mathrm dx\,\mathrm dt+\int_0^T\int_{\mathbb{R}^n}\left(2|\nabla \varphi(x)|^2-\varphi(x)\triangle \varphi(x)\right) u^2(t,x)\,\mathrm dx\,\mathrm dt.\nonumber
\end{eqnarray}
Since the support of $u_0$ is contained in $B_r$, we have that $v(0,\cdot)=0$ over $\mathbb{R}^n$.
Then by the H\"older inequality, we deduce from \eqref{equ-special-1} that
\begin{align}\label{equ-special-2}
\int_{\mathbb{R}^n}v^2(T,x)\,\mathrm dx \leq \int_0^T\int_{\mathbb{R}^n}2\left(3|\nabla \varphi(x)|^2-\varphi(x)\triangle \varphi(x)\right) u^2(s,x)\,\mathrm dx\,\mathrm ds.
\end{align}
Note that \eqref{equ-special-2} is still true if we replace  $T$ by any $t\in (0,T)$. This implies that
\begin{align}\label{equ-special-3}
\int_0^T\int_{\mathbb{R}^n}v^2(t,x)\,\mathrm dx\,\mathrm dt \leq \int_0^T\int_0^t\int_{\mathbb{R}^n}2\left(3|\nabla \varphi(x)|^2-\varphi(x)\triangle \varphi(x)\right) u^2(t,x)\,\mathrm dx\,\mathrm d\tau \,\mathrm dt.
\end{align}
Since $v=u$ on $B_M^c$, it follows from (\ref{3.63GGwangg})
 and  \eqref{equ-special-3} that
\begin{align}\label{equ-special-4}
\int_0^T\int_{|x|\geq M}u^2(t,x)\,\mathrm dx \leq \frac{CnT}{(M-r)^2}\int_0^T\int_{r\leq |x|\leq M}u^2(t,x)\,\mathrm dx \,\mathrm dt.
\end{align}
Meanwhile, it is clear that
\begin{align}\label{equ-special-5}
\int_{\mathbb{R}^n}u^2(T,x)\,\mathrm dx \leq \frac{1}{T}\int_0^T\int_{\mathbb{R}^n}u^2(t,x)\,\mathrm dx \,\mathrm dt.
\end{align}
Combining \eqref{equ-special-4} and \eqref{equ-special-5}, we obtain that
 \begin{align*}
\int_{\mathbb{R}^n}u^2(T,x)\,\mathrm dx &\leq \frac{1}{T}\int_0^T\int_{|x|\leq M}u^2(t,x)\,\mathrm dx \,\mathrm dt+ \frac{1}{T}\int_0^T\int_{|x|\geq M}u^2(t,x)\,\mathrm dx \,\mathrm dt\nonumber\\
&\leq \frac{1}{T}\int_0^T\int_{|x|\leq M}u^2(t,x)\,\mathrm dx \,\mathrm dt+ \frac{1}{T} \frac{CnT}{(M-r)^2}\int_0^T\int_{r\leq |x|\leq M}u^2(t,x)\,\mathrm dx \,\mathrm dt\nonumber\\
&\leq \left(\frac{1}{T}+\frac{Cn}{(M-r)^2}\right)\int_0^T\int_{|x|\leq M}u^2(t,x)\,\mathrm dx \,\mathrm dt.
\end{align*}
which leads to the conclusion (i) of Theorem \ref{obs-special-data}.

\medskip

(ii)  Let $T>0$ and $M>0$ be arbitrarily given.
Arbitrarily fix $u_0$ so that
 \begin{eqnarray}\label{3.70WWGGSS}
 0\leq u_0\in L^1(\mathbb{R}^n);\;\; \int_{B_r}u_0(x)\mathrm dx\geq\mu\int_{\mathbb{R}^n}u_0(x)\mathrm dx\;\; \mbox{for some}\;\;r>0\;\;\mbox{and}\;\;  \mu \in (0,1).
 \end{eqnarray}
 Write $u$ for the solution to \eqref{heat} with $u(0,\cdot)=u_0(\cdot)$.

 We first prove that when $0<M\leq r$,
\begin{align}\label{equ-10-29-1}
\int_{\mathbb{R}^n} |u(T,x)|^2\,\mathrm dx \leq \frac{2^{\frac{n}{2}+1}\pi^{\frac{n}{2}}T^{\frac{n}{2}-1}}{V_n M^n \mu^2}e^{\frac{4r^2}{T}}\int_{0}^T \int_{B_M} |u(t,x)|^2\,\mathrm dx\,\mathrm dt.
\end{align}
For this purpose,  we need the following two estimates:
\begin{align}\label{equ-10-29-2}
\int_{\frac{T}{2}}^T \int_{|x|\leq M} u^2(t,x)\,\mathrm dx\,\mathrm dt \geq 2^{-1}(4\pi)^{-n}V_n M^n\mu^2T^{-(n-1)}e^{-\frac{4r^2}{T}}\left(\int_{\mathbb{R}^n} u_0(x)\,\mathrm dx\right)^2;
\end{align}
\begin{align}\label{equ-10-29-3}
\int_{\mathbb{R}^n} u^2(T,x)\,\mathrm dx \leq  2^{-\frac{3n}{2}}(\pi T)^{-n/2}\left(\int_{\mathbb{R}^n} u_0(x)\,\mathrm dx\right)^2.
\end{align}
To show \eqref{equ-10-29-2}, we observe that
\begin{align}\label{heat-kernel}
u(t,x) = \int_{\mathbb{R}^n}(4\pi t)^{-n/2}e^{-\frac{|x-y|^2}{4t}}u_0(y)\,\mathrm dy,\quad (t,x)\in\mathbb (0,\infty) \times\mathbb R^n.
\end{align}
By (\ref{heat-kernel}) and (\ref{3.70WWGGSS}), we find that when $t>0,\ |x|\leq r$,
\begin{align}\label{equ-special-15}
u(t,x)&\geq \int_{|y|\leq r}(4\pi t)^{-n/2}e^{-\frac{|x-y|^2}{4t}}u_0(y)\,\mathrm dy\geq \int_{|y|\leq r}(4\pi t)^{-n/2}e^{-\frac{r^2}{t}}u_0(y)\,\mathrm dy\nonumber\\
&\geq \mu(4\pi t)^{-n/2}e^{-\frac{r^2}{t}}\int_{\mathbb{R}^n} u_0(x)\,\mathrm dx.
\end{align}
Since $M\leq r$, it follows from \eqref{equ-special-15} that
\begin{align*}
\int_{\frac{T}{2}}^T \int_{|x|\leq M} u^2(t,x)\,\mathrm dx\,\mathrm dt &\geq V_n M^n\mu^2\left(\int_{\mathbb{R}^n} u_0(x)\,\mathrm dx\right)^2\int_{\frac{T}{2}}^T(4\pi t)^{-n}e^{-\frac{2r^2}{t}}\,\mathrm dt\\
&\geq 2^{-1}(4\pi)^{-n}V_n M^n\mu^2T^{-(n-1)}e^{-\frac{4r^2}{T}}\left(\int_{\mathbb{R}^n} u_0(x)\,\mathrm dx\right)^2,
\end{align*}
which leads to \eqref{equ-10-29-2}.

We now  show \eqref{equ-10-29-3}. By (\ref{heat-kernel}) and
 the Young inequality,  we have that
\begin{align*}
\|u(T,x)\|_{L^2(\mathbb{R}^n)}&\leq \Big\|(4\pi T)^{-n/2}e^{-\frac{|x|^2}{4T}}\Big\|_{L^2(\mathbb{R}^n)}\int_{\mathbb{R}^n} u_0(x)\,\mathrm dx=2^{-\frac{3n}{4}}(\pi T)^{-n/4}\int_{\mathbb{R}^n} u_0(x)\,\mathrm dx,
\end{align*}
which leads to \eqref{equ-10-29-3}.

Next, by \eqref{equ-10-29-2} and \eqref{equ-10-29-3}, we see that
\begin{align*}
\int_{\mathbb{R}^n} u^2(T,x)\,\mathrm dx&\leq \frac{2^{\frac{n}{2}+1}\pi^{\frac{n}{2}}T^{\frac{n}{2}-1}}{V_n M^n \mu^2}e^{\frac{4r^2}{T}}\int_{\frac{T}{2}}^T \int_{|x|\leq M} u^2(t,x)\,\mathrm dx\,\mathrm dt\\
&\leq \frac{2^{\frac{n}{2}+1}\pi^{\frac{n}{2}}T^{\frac{n}{2}-1}}{V_n M^n \mu^2}e^{\frac{4r^2}{T}}\int_{0}^T \int_{|x|\leq M} u^2(t,x)\,\mathrm dx\,\mathrm dt,
\end{align*}
which leads to  \eqref{equ-10-29-1} for the case when $0<M\leq r$.

Finally, when $M>r$, we apply \eqref{equ-10-29-1} (with $M=r$) to obtain that
\begin{align*}
\int_{\mathbb{R}^n} u^2(T,x)\,\mathrm dx &\leq \frac{2^{\frac{n}{2}+1}\pi^{\frac{n}{2}}T^{\frac{n}{2}-1}}{V_n r^n \mu^2}e^{\frac{4r^2}{T}}\int_{0}^T \int_{|x|\leq r} u^2(t,x)\,\mathrm dx\,\mathrm dt\\
&\leq \frac{2^{\frac{n}{2}+1}\pi^{\frac{n}{2}}T^{\frac{n}{2}-1}}{V_n r^n \mu^2}e^{\frac{4r^2}{T}}\int_{0}^T \int_{|x|\leq M} u^2(t,x)\,\mathrm dx\,\mathrm dt,
\end{align*}
which leads to \eqref{equ-10-29-1}  for the case that $M>r$. So the conclusion (ii)
in Theorem \ref{obs-special-data} is true.

Hence, we end the proof of Theorem \ref{obs-special-data}
\end{proof}

\medskip

\textbf{Acknowledgment}.
This work was partially supported  by the National Natural Science Foundation of China under grants 11501424 and 11701535.

\end{document}